\documentclass[12pt,a4paper]{amsart}

\usepackage{amssymb}
\usepackage{amscd}
\usepackage{color}

\makeatletter
\@namedef{subjclassname@2020}{\textup{2020} Mathematics Subject Classification}
\makeatother

\def\frak{\mathfrak}
\def\Bbb{\mathbb}
\def\Cal{\mathcal}

\let\phi\varphi

\newcommand{\tr}{\operatorname{tr}}

\newcommand{\x}{\times}
\renewcommand{\o}{\circ}

\newcommand{\al}{\alpha}

\newcommand{\ep}{\epsilon}

\newcommand{\la}{\lambda}
\newcommand{\om}{\omega}
\newcommand{\ph}{\phi}
\newcommand{\ps}{\psi}

\newcommand{\si}{\sigma}
\newcommand{\ze}{\zeta}
\newcommand{\Ga}{\Gamma}
\newcommand{\La}{\Lambda}
\newcommand{\Ph}{\Phi}
\newcommand{\Ps}{\Psi}
\newcommand{\Om}{\Omega}

\newcommand{\Up}{\Upsilon}
\def\Rho{\mbox{\textsf{P}}}
\newcommand{\bep}{\boldsymbol{\epsilon}}

\newcommand{\barm}{\overline{M}}
\newcommand{\tPsi}{\widetilde{\Psi}}
\newcommand{\vol}{\operatorname{vol}}
\newcommand{\im}{\operatorname{im}}
\newcommand{\id}{\operatorname{id}}

\newcommand{\Diff}{\operatorname{Diff}}
\newcommand{\Conf}{\operatorname{Conf}}
\newcommand{\Fl}{\operatorname{Fl}}

\newcommand{\rpl}                         
{\mbox{$
\begin{picture}(12.7,8)(-.5,-1)
\put(0,0.2){$+$}
\put(4.2,2.8){\oval(8,8)[r]}
\end{picture}$}}

\numberwithin{equation}{section}

\newcounter{theorem}
\numberwithin{theorem}{section}
\newtheorem{thm}[theorem]{Theorem}
\newtheorem*{thm*}{Theorem \thesubsection}
\newtheorem{lemma}[theorem]{Lemma}
\newtheorem{prop}[theorem]{Proposition}
\newtheorem{cor}[theorem]{Corollary}
\newtheorem*{lemma*}{Lemma \thesubsection}
\newtheorem*{prop*}{Proposition \thesubsection}
\newtheorem*{cor*}{Corollary \thesubsection}

\theoremstyle{definition}
\newtheorem{definition}[theorem]{Definition}
\newtheorem*{definition*}{Definition \thesubsection}

\newtheorem*{example*}{Example \thesubsection}
\theoremstyle{remark}
\newtheorem{remark}[theorem]{Remark}
\newtheorem*{remark*}{Remark \thesubsection}

\def\sideremark#1{\ifvmode\leavevmode\fi\vadjust{\vbox to0pt{\vss
 \hbox to 0pt{\hskip\hsize\hskip1em
 \vbox{\hsize3cm\tiny\raggedright\pretolerance10000
  \noindent #1\hfill}\hss}\vbox to8pt{\vfil}\vss}}}%
                        
\begin{document}
\title{A boundary-local mass cocycle and the mass of asymptotically
  hyperbolic manifolds}

\author{Andreas \v Cap, and A.\ Rod Gover}

\address{A.\v C.: Faculty of Mathematics\\
University of Vienna\\
Oskar--Morgenstern--Platz 1\\
1090 Wien\\
Austria\\
A.R.G.:Department of Mathematics\\
  The University of Auckland\\
  Private Bag 92019\\
  Auckland 1142\\
  New Zealand}
\email{Andreas.Cap@univie.ac.at}
\email{r.gover@auckland.ac.nz}

\begin{abstract}
We construct a cocycle that, for a given $n$-manifold, maps pairs of
asymptotically locally hyperbolic (ALH) metrics to a tractor-valued
$(n-1)$-form field on the conformal infinity.
This
requires the metrics to be asymptotically related to a given order
that depends on the dimension. It then provides a local geometric
quantity on the boundary that is naturally associated to the pair and
can be interpreted as a relative energy-momentum density.
It is distinguished as a
geometric object by its property of being invariant under suitable
diffeomorphisms fixing the boundary, and that act on
(either) one of the argument metrics.

Specialising to the case of an ALH metric $h$ that is suitably
asymptotically related to a locally hyperbolic conformally compact
metric, we show that the cocycle determines an absolute invariant
$c(h)$, which still is
local in nature. This tractor-valued $(n-1)$-form field on the conformal
infinity is canonically associated to $h$ (i.e. is not dependent on
other choices) and is equivariant under the appropriate
diffeomorphisms.

Finally specialising further to the case that the boundary is a sphere and
that a metric $h$ is asymptotically related to a hyperbolic metric on
the interior, we show that the invariant $c(h)$ can be integrated
over the boundary. The result pairs with solutions of the KID (Killing
initial data) equation to recover the known description of hyperbolic
mass integrals of Wang, and Chru\'{s}ciel--Herzlich.
  \end{abstract}


\thanks{A.\v C.\ gratefully acknowledges support by the Austrian Science Fund (FWF):
  grant DOI 10.55776/P33559 and the hospitality of the University of Auckland. A.R.G.\ gratefully
  acknowledges support from the Royal Society of New Zealand via Marsden Grants
  16-UOA-051 and 19-UOA-008. We thank P.\ Chru\'sciel for very helpful
  discussions. We also thank the anonymous referee for several
    suggestions which helped to improve the article.}

\maketitle

\pagestyle{myheadings} \markboth{\v Cap, Gover}{Mass and Tractors}

\section{Introduction}\label{1}

In general relativity, and a number of related mathematical studies,
the notion of a ``mass'' invariant for relevant geometric manifolds is
extremely important and heavily studied \cite{Bartnik,LP,S-Yau}.  In
general defining and interpreting a suitable notion of mass is not
straightforward. For so-called asymptotically flat manifolds, the
Arnowitt-Deser-Misner (ADM) energy-momentum is well established and is
usually accepted as the correct definition.  Motivated by the desire
to define and study mass in other settings X.\ Wang \cite{Wang} and
P.T.\ Chru\'sciel and M.\ Herzlich \cite{Chrusciel-Herzlich}
introduced a notion of mass integrals and ``energy-momentum'' for
Riemannian manifolds that, in a suitable sense, are asymptotically
hyperbolic. These have immediate applications for a class of static
spacetimes. The mass integrals are parametrized by solutions of the
KID (Killing initial data) equation for the hyperbolic background.
Denoting by $n$ the interior dimension, these solutions form an
$n+1$-dimensional vector space that is endowed with a natural
Lorentzian metric. This leads to the interpretation that the mass in
the asymptotically hyperbolic setting is a vector in that
$n+1$-dimensional space rather than just a number.

The aim of this article is to construct new invariants that capture a
notion of mass density, in the setting of asymptotically hyperbolic
metrics.
These invariants are locally defined (as quantities on the boundary at
infinity) volume forms with values in the standard tractor bundle
associated to a conformal structure on the boundary, which is has rank
$n+1$ and carries a natural Lorentzian bundle metric.  In the special
case of hyperbolic space as a background, we prove that these local
quantities can be integrated to global parallel sections of the
tractor bundle and then naturally recover the mass as introduced by
Wang and Chru\'sciel--Herzlich.
The setting
we work in is rather restrictive in some aspects but very general in
other aspects. The main restriction is that we are working in a
conformally compact setting, so we need strong assumptions on the
order of asymptotics.  On the other hand, underlying this is an
arbitrary manifold with boundary with no restrictions on the topology
of the boundary. We also allow a fairly general ``background metric'';
the core of our results only require a background metric that is
asymptotically locally hyperbolic (ALH), as in Definition
\ref{def2.2}, part (1).  The main invariant we construct is associated
to a pair of ALH metrics (that are asymptotic to sufficient order), so
it should be thought of as a \textit{relative local mass} or as a
\textit{local mass difference}.

Our basic setting looks as follows. We start with an arbitrary
manifold $\barm$ of dimension $n\geq 3$, with boundary $\partial M$
and interior $M$. Given two conformally compact metrics $g$ and $h$ on
$M$, there is a well defined notion of $g$ and $h$ approaching each
other asymptotically to certain orders towards the boundary. The
actual order we need depends on the dimension and specializes to the
order required in \cite{Wang} on hyperbolic space. This defines an
equivalence relation on conformally compact metrics and we consider
one equivalence class $\Cal G$ of such metrics. The only additional
requirement at this point is that $\Cal G$ consists of
ALH-metrics. Since $\Cal G$ consists of conformally compact metrics,
each $g\in\Cal G$ gives rise to a conformal structure on $\partial M$,
called the conformal infinity of $g$. Moreover, the asymptotic
condition used to define $\Cal G$ is strong enough to ensure that all
metrics in $\Cal G$ lead to the same conformal infinity. Thus $\Cal G$
canonically determines a conformal structure $[\Cal G_\infty]$ on $\partial
M$.

This last fact is crucial for the
further development, since the invariants we construct are geometric
objects for this conformal structure on $\partial M$ (and a slightly
stronger structure in case that $\dim(\barm)=3$). Indeed, the conformal
structure $[\Cal G_\infty]$ canonically determines the so-called
\textit{standard tractor bundle} $\Cal T\partial M\to\partial M$. This
is a vector bundle of rank $n+1$ endowed with a Lorentzian bundle
metric and a metric linear connection \cite{BEG}. We construct
\textit{cocycles} $c$ that associate to each pair of metrics
$g,h\in\Cal G$ a tractor-valued $(n-1)$-form
$c(g,h)\in\Om^{n-1}(\partial M,\Cal T\partial M)$, with the property
that $c(h,g)=-c(g,h)$ and $c(g,k)=c(g,h)+c(h,k)$ for all $g,h,k\in\Cal
G$. These cocycles are the basic ``relative local masses'' we
consider.

The idea of the construction is that any metric $g\in\Cal G$
determines a conformal class on $\barm$ (since the conformal class of
$g$ on $M$ extends to the boundary by conformal compactness) and this
restricts to $[\Cal G_\infty]$ at the boundary. Working with this
conformal structure, we can use standard techniques of tractor
calculus, and an instance of what is called a BGG splitting operator,
to associate to any $h\in\Cal G$ a one-form with values in the tractor
bundle $\Cal T\barm$. The Hodge dual (with respect to $g$) of this
one-form is then shown to admit a smooth extension of the boundary
whose boundary value is defined to be $c(g,h)$, see Propositions
\ref{prop3.1} and \ref{prop3.2}. The proof that this actually defines
a cocycle needs some care because of the use of different conformal
structures on $\barm$, but is otherwise straightforward.  The idea to
seek objects valued in $\Cal T\barm$ was motivated by a desire to link
to the KID (Killing initial data) equation, and its solutions, which
(as we discuss later), has a nice interpretation in the tractor
picture. This, with the conformal compactification, leads to the need
to use conformal tractors and tools that are conformally invariant so
that they extend in a simple and practical way to the boundary. The
BGG splitting operators we use are then the unique conformally
invariant operators available to extract the required jet data and
include this is the required bundles.  The use of tractors and
conformally invariant operators is also crucial for allowing us to
work locally (along the boundary), which is an essential feature of
our approach.

Initially, this leads to a two parameter family of cocycles since there
are two constructions of the above type, one depending on the trace of
the difference $g-h$, the other on its trace-free part. Our
constructions do not require charts or coordinates, so, in that sense,
 are automatically geometric in nature. The constructions also
readily imply equivariancy with respect to the appropriate diffeomorphisms. If
$\Ph$ is a diffeomorphism of $\barm$ which preserves the class $\Cal
G$ (in an obvious sense) then by definition the restriction
$\Ph|_{\partial M}$ is a conformal isometry for $[\Cal G_\infty]$. Therefore,
it naturally acts on sections of $\Cal T\partial M$ (and of course on 
forms on $\partial M$) and for any of the cocycles
$c(\Phi^*g,\Phi^*h)=(\Phi|_{\partial M})^*c(g,h)$, see Proposition
\ref{prop3.3}.

There is a much more subtle compatibility condition with
diffeomorphisms, however. Indeed, consider a diffeomorphism $\Psi$
that is compatible with $[\Cal G_\infty]$ and suppose that $\Psi|_{\partial
  M}=\id_{\partial M}$. Then it turns out that $\Psi$ is asymptotic to
the identity of order $n+1$ in a well-defined sense, see Section
\ref{3.4} and Theorem \ref{thm3.5}. The main technical result of our
article then is that there is a unique ratio of the two parameters for
our cocycles, which ensures that $c(g,\Psi^*h)=c(g,h)$ for any $\Psi$
which is asymptotic to the identity of order $n+1$. Hence, up to an
overall normalization we obtain a unique cocycle $c$ which has this
invariance property in addition to the equivariancy property mentioned
above. This proof is based on an idea in \cite{CDG} that shows that
the action of diffeomorphisms that are asymptotic to the identity can
be absorbed into a geometric condition relating the two metrics (and
an adapted defining function for one of them). The key feature of this
property is that it holds in the general setting of a class of ALH
metrics on a manifold with boundary.  In the special cases for which
we obtain invariants of single metrics, this property enables us to
prove equivariancy of such invariants under diffeomorphisms preserving
$\Cal G$, see below.

To pass from our cocycles to invariants of a single metric, one has to
go to specific situations in which $\Cal G$ contains particularly
nice metrics. We only discuss the case that $\Cal G$ locally contains
metrics that are hyperbolic, i.e.\ have constant sectional curvature
$-1$. This of course implies that $(\partial M,[\Cal G_\infty])$ is
conformally flat, but it does not impose further restrictions on the
topology of $\partial M$, see Section \ref{3.8}. Under this
assumption, we show that, for a cocycle from our one-parameter family
and a metric $h\in\Cal G$, all local hyperbolic metrics $g\in\Cal G$
locally lead to the same tractor-valued form $c(g,h)$. These local
forms then piece together to define an object
$c(h)\in\Om^{n-1}(\partial M,\Cal T\partial M)$ that is canonically
associated to $h$. We prove that this is equivariant under
diffeomorphisms preserving the class $\Cal G$, see Theorem
\ref{thm3.8}.

As a last step, we prove that, in the conformally compact setting, the AH mass as
introduced in \cite{Chrusciel-Herzlich} can be obtained by integrating our
invariant. Thus we have to specialize to the case that $\barm$ is an open
neighborhood of the boundary sphere in the closed unit ball and $\Cal G$ contains the
restriction of the Poincar\'e metric. This implies that $\partial M=S^{n-1}$ and
$[\Cal G_\infty]$ is the round conformal structure and thus the standard tractor bundle
$\Cal TS^{n-1}$ can be globally trivialized by parallel sections. Hence $\Cal
TS^{n-1}$--valued $(n-1)$-forms can be integrated to global parallel sections of
$\Cal TS^{n-1}$ (say by expanding in any globally parallel frame and then integrating
the coefficient forms). Now it is well known how to make the trivialization of $\Cal
TS^{n-1}$ explicit, and we show that there is a particularly nice way to do this
using the conformal class $[g]$ on $\barm$ determined by the Poincar\'e metric. Via
boundary values of parallel tractors in the interior, the parallel sections of $\Cal
TS^{n-1}$ turn out to be parametrized by the solutions of the KID (Killing initial
data) equation (\ref{KID}) on $M$. But these solutions exactly parameterize the mass
integrals used to define the AH mass in the style of \cite{Chrusciel-Herzlich}, see
\cite{Michel}. This last fact was our original motivation to look for a tractor
interpretation of the AH mass.  

Now a solution $V$ of the KID equation determines a parallel section
$s_V$ of $\Cal TS^{n-1}$ and we can proceed as follows. Given
$h\in\Cal G$, we can form the invariant $c(h)\in
\Om^{n-1}(S^{n-1},\Cal TS^{n-1})$ and integrate it to a parallel
section $\int_{S^{n-1}}c(h)$ of $\Cal TS^{n-1}$. This can then be
paired, via the tractor metric, with $s_V$. Analyzing the boundary
behavior of the mass integral determined by $V$, we show in Theorem
\ref{thm3.9} that, after appropriate normalization, this pairing
exactly recovers the mass integral, of \cite{Chrusciel-Herzlich},
determined by $V$.

Throughout all manifolds, tensor fields, and related objects, will be
taken to be smooth in the sense of $C^\infty$. For most results lower
regularity would be sufficient, but we do not take that up here.

\section{Setup and tractors}\label{2}

\subsection{Conformally compact metrics}\label{2.1}
Let $\barm$ be a smooth manifold with boundary $\partial M$ and
interior $M$. Recall that a \textit{local defining function} for
$\partial M$ is a smooth function $\rho:U\to\Bbb R_{\geq 0}$ defined
on some open subset $U\subset\barm$ such that $U\cap\partial
M=\rho^{-1}(\{0\})$ and such that $d\rho|_{U\cap\partial M}$ is
nowhere vanishing. For two local defining functions $\rho$ and
$\hat\rho$ defined on the same open set $U$, there is a smooth
function $f:U\to\Bbb R_{>0}$ such that $\hat\rho=f\rho$. It is often
convenient to write such positive function as $f=e^{\tilde f}$ for
some smooth function $\tilde f$. This notion of defining functions
extends, without problem, from functions to smooth sections of line
bundles. One just has to replace $d\rho$ by the covariant derivative
with respect to any linear connection, which is independent of the
connection along the zero set of the section. In particular, taking the line bundle
concerned to be a density bundle (as discussed above Proposition \ref{prop2.1} below)
leads to \textit{defining densities} \cite{Proj-comp}.  

A pseudo-Riemannian metric $g$ on $M$ is called \textit{conformally compact} if, for
any point $x\in\partial M$, there is a local defining function $\rho$ for $\partial
M$ defined on some neighborhood $U$ of $x$ such that $\rho^2g$ admits a
(sufficiently) smooth extension from $U\cap M$ to all of $U$, with
  non-degenerate boundary values. It is easy to see that this property is independent
of the choice of defining function, so if $g$ is conformally compact, then for any
local defining function $\hat\rho$ for $\partial M$, the metric $\hat\rho^2 g$ admits
a smooth extension to the boundary. While the metric on the boundary, induced by such
an extension, depends on the chosen defining function, such extensions are always
conformally related. Hence a conformally compact metric on $M$ gives rise to a
well-defined conformal class on $\partial M$, which is called the \textit{conformal
  infinity} of $g$. The model example of a conformally compact metric is the
Poincar\'e ball model of hyperbolic space. Here $\barm$ is the closed unit ball in
$\Bbb R^{n}$ and one defines a metric $g$ on the open unit ball as
$\tfrac{4}{(1-|x|^2)^2}$ times the Euclidean metric, see \cite{Graham:Srni}. The
resulting conformal infinity is the conformal class of the round metric on $S^{n-1}$.

There is a conceptual description of conformally compact metrics in the language of
conformal geometry: Since $\rho^2g$ is conformal to $g$ away from the boundary, it
provides a (sufficiently) smooth extension of the conformal structure on $M$ defined
by $g$ to all of $\barm$. Conformally compact metrics can be neatly characterized in
this picture via their volume densities. Recall that the volume density of $g$ is
nowhere vanishing, so one can form powers with arbitrary real exponents to obtain
nowhere vanishing densities of all (non-zero) weights. Each of these densities is
parallel for the connection on the appropriate density bundle induced by the
Levi-Civita connection, and up to constant multiples, this is the only parallel
section.

In the usual conventions of conformal geometry, see \cite{BEG}, the square of the top
exterior power of the tangent bundle, that is $\otimes^2(\Lambda^nTM)$, is identified
with a line bundle of {\em densities} of weight $2n$ that we denote
$\mathcal{E}[2n]$. This is oriented and hence trivial, so there is a
standard notion of its roots. With these conventions, on an $n$-manifold a metric $g$
determines a volume density $\vol_g$ that has conformal weight $-n$, meaning
$\vol_g\in\Ga(\Cal E[-n])$. Rescaling $g$ to $\rho^2g$, we get
$\vol_{\rho^2g}=\rho^n\vol_g$ and thus
$\vol_{g}^{-1/n}=\rho\vol_{\rho^2g}^{-1/n}\in\Ga(\Cal E[1])$. Assuming that $g$ is
conformally compact, $\vol_{\rho^2g}^{-1/n}$ is smooth up to the boundary and nowhere
vanishing, so this equation shows that $\vol_{g}^{-1/n}\in\Ga(\Cal E[1])$ is a
defining density for $\partial M$. Similar arguments prove the converse, which leads
to the following result.
\begin{prop}\label{prop2.1}
Let $\barm$ be a smooth manifold with boundary $\partial M$ and interior $M$, which
is endowed with a conformal structure $c$. Then a metric $g$ on $M$ which lies in
$c|_M$ is conformally compact if an only if any non-zero section $\si\in\Ga(\Cal
E[1]|_M)$, which is parallel for the Levi-Civita connection of $g$ extends by $0$ to a
defining density for $\partial M$.
\end{prop}

\subsection{ALH-metrics and adapted defining functions}\label{2.2}
Consider a conformally compact metric $g$ and for a local defining
function $\rho$ put $\overline{g}:=\rho^2g$, and note that this is
smooth and non-degenerate up to the boundary. Hence also the inverse metric
$\overline{g}^{-1}$ is smooth up to the boundary, so
$\overline{g}^{-1}(d\rho,d\rho)$ a smooth function on the domain of
definition of $\rho$. Moreover, since $d\rho$ is nowhere vanishing
along the boundary $\overline{g}^{-1}(d\rho,d\rho)$ has the same
property. Replacing $\rho$ by $\hat\rho:=e^f\rho$, we obtain
$d\hat\rho=\hat\rho df+e^fd\rho$, which easily implies that the
restriction of $\overline{g}^{-1}(d\rho,d\rho)$ to $\partial M$ is
independent of the defining function $\rho$.

So $\overline{g}^{-1}(d\rho,d\rho)$ is an invariant of the metric $g$,
and this is related to the asymptotic behavior of the curvature of
$g$, see e.g.\ \cite{Graham:Srni}. In particular, if
$\overline{g}^{-1}(d\rho,d\rho)|_{\partial M}\equiv 1$, then the
sectional curvature of $g$ is asymptotically constant $-1$. This
justifies the terminology in the first part of the following
definition and leads to a subclass of defining functions:
\begin{definition}\label{def2.2}
  Consider a smooth manifold $\barm=M\cup\partial M$ with boundary and a conformally
  compact metric $g$ on $M$. 

  (1) The metric $g$ is called \textit{asymptotically locally hyperbolic} (or an
  ALH-metric) if $(\rho^2g)^{-1}(d\rho,d\rho)|_{\partial M}$ is identically one.

  (2) Assume that $g$ is an ALH-metric on $M$ and that $U\subset\barm$ (with
  $U\cap\partial M\neq\emptyset$ to be of interest). Then a local defining function
  $\rho$ for $\partial M$ defined on $U$ is called \textit{adapted to $g$} if the
  function $(\rho^2 g)^{-1}(d\rho,d\rho)$ is identically one on some open
  neighborhood of $U\cap\partial M$.
\end{definition}

\begin{remark} Note that in the literature the terminology
  ``asymptotically locally hyperbolic'' is sometimes used for (various) more
  restrictive classes of geometry. Also the condition we have (1) is
  sometimes referred to as simply ``asymptotically
  hyperbolic''. However the latter is also used for rather special
  settings where in particular the boundary is necessarily a sphere,
  hence our use here.
\end{remark}

The existence of adapted defining functions, as in part (2) of
Definition \ref{def2.2}, can be established by solving an appropriate
non-characteristic first order PDE. The following precise description
of adapted defining functions is given in Lemma 2.1 of
\cite{Graham:Srni}.
\begin{prop}\label{prop2.2}
  Consider $\barm=M\cup\partial M$ and an ALH metric $g$ on $M$. Then
  for any choice of a representative metric $h$ in the conformal
  infinity of $g$, there exists an adapted defining function $\rho$ for
  $g$, defined on an open neighborhood of $\partial M$ in $\barm$,
  such that $\rho^2g$ induces the metric $h$ on $\partial
  M$. Moreover, for fixed $h$, the germ of $\rho$ along $\partial M$
  is uniquely determined.
\end{prop}

\subsection{The basic setup}\label{2.3}
Defining functions can be used to measure the asymptotic growth (or fall-off) of
functions and more general geometric objects on the interior of a manifold with
boundary. A fundamental property of defining functions is that for a function $f$
that is smooth up to the boundary $f|_{\partial M}\equiv 0$ if and only if for any
local defining function $\rho$ for $\partial M$, we obtain, on the domain of
definition of $\rho$, $f=\rho f_1$ for a function $f_1$ that is smooth up to the
boundary. We say that $f$ is $\Cal O(\rho)$ in this case, observing that this notion
is actually independent of the specific defining function $\rho$. Similarly if, in
such an expansion, $f_1$ also vanishes along the boundary then this fact does not
depend on the choice of $\rho$, and in that case we say that $f$ is $\Cal
O(\rho^2)$. Inductively, one obtains the notion that $f$ is $\Cal O(\rho^N)$ for any
integer $N>0$, which again does not depend on the specific choice of defining
function.

Given a smooth function $f:M\to\Bbb R$ and an integer $N>0$, we then say that $f$ is
${\Cal O} (\rho^{-N})$ if locally around each $x\in \partial M$ we find a defining
function $\rho$ such that $\rho^Nf$ admits a smooth extension to the boundary. Again,
the fact that such an extension exists is independent of the choice of defining
function, as is vanishing of the boundary value of $\rho^N f$ in some point. In
points where the boundary value of $\rho^Nf$ is nonzero, the actual value does depend
on the choice of $\rho$, however.

For two functions $f_1,f_2:\barm\to\Bbb R$ and $N>0$, we write $f_1\sim_N f_2$ if
$f_1-f_2$ is $\Cal O(\rho^N)$. By definition, this means that, on the domain of a
local defining function $\rho$, we can write $f_2=f_1+\rho^Nf$ for some function $f$
that is smooth up to the boundary. Observe that this defines an equivalence relation.

All this extends without problems to tensor fields of arbitrary
(fixed) type. This can be seen immediately from looking at coordinate
functions in (boundary) charts. So for $N>0$, a tensor field $t$ on
$M$ is $\Cal O(\rho^N)$ if can be written as $\rho^N\tilde t$ for a
tensor field $\tilde t$ that is smooth up to the boundary. Likewise,
$t$ is $\Cal O(\rho^{-N})$ if $\rho^Nt$ admits a smooth extension to
the boundary. In the obvious way we extend the notation $t_1\sim_N
t_2$ with $N>0$ to tensor fields $t_1,t_2$ that are smooth up to the
boundary. Observe also that these concepts are compatible in an
obvious sense with tensorial operations, like inserting vector fields
into metrics, etc.

In this language, a conformally compact metric $g$ is $\Cal O(\rho^{-2})$ and,
writing $g=\rho^{-2}\overline g$, it satisfies that $\overline g$ is nowhere
vanishing and non-degenerate along the boundary. Observe that this
implies that the inverse metric $g^{-1}$ is $\Cal O(\rho^2)$ and, in particular,
vanishes along $\partial M$. For conformally compact metrics $g$ and $h$ we can
consider the metrics $\rho^2g$ and $\rho^2h$ that are smooth up to the boundary and
require that $\rho^2g\sim_N\rho^2h$, which again is independent of the choice of
defining function $\rho$. This defines an equivalence relation on the set of
conformally compact metrics, and to simplify notation, we formally write this as
$g\sim_{N-2}h$. In the current article we will mainly be interested in the case that
$N=n=\dim(M)$ but we carry out most computations for general integers $N>0$, since
this does not lead to difficulties and as a preparation for later extensions.

We will start with an equivalence class $\Cal G$ of conformally compact metrics on
$M$ with respect to the equivalence relation $\sim_{N-2}$ for some $N>0$. Observe
that for two metrics $g,h\in\Cal G$ and a local defining function $\rho$, the metrics
$\rho^2g$ and $\rho^2h$, by definition, admit a smooth extension to the boundary with
the same boundary value. In particular, they induce the same conformal infinity on
$\partial M$. Thus the class $\Cal G$ of metrics gives rise to a well defined
conformal structure on $\partial M$ that we will denote by $[\Cal G_\infty]$. As we shall
see below, if one metric $g\in\Cal G$ is ALH, then the same holds for all metrics in
$\Cal G$. We shall always assume that this is the case from now on. Let us also
remark here that, in the case that $\dim(\barm)=3$, $[\Cal G_\infty]$ induces a stronger
structure that just a conformal structure on $\partial M$ which will be needed in
what follows. This will be discussed in more detail below.

From now on, we will sometimes use abstract index notation. In that
notation we write $g_{ij}$ for the metric $g$, $g^{ij}$ for its
inverse and so on (even though no coordinates or frame field is
chosen). Given $g_{ij},h_{ij}\in\Cal G$, and fixing a local defining
function $\rho$ for $\partial M$,  by definition there is a section
$\mu_{ij}$ that is smooth up to the boundary such that
\begin{equation}\label{h-g}
   h_{ij}=g_{ij}+\rho^{N-2}\mu_{ij}. 
\end{equation}
Since $g^{ij}$ is $\Cal O(\rho^2)$, we see that $g^{ij}\mu_{ij}=\rho^2\mu$ for some
function $\mu$ that is smooth up to the boundary, whence
$g^{ij}(h_{ij}-g_{ij})=\rho^N\mu$. Using this, we next compute the relation between
the defining densities $\si,\tau\in\Ga(\Cal E[1])$ determined by $g_{ij}$ and
$h_{ij}$, respectively. Writing
\begin{equation}\label{h-g2}
  h_{ij}=g_{ik}(\delta^k_j+\rho^{N-2}g^{k\ell}\mu_{\ell j}),
\end{equation}
we can take determinants to find that
$\det(h_{ij})=\det(g_{ik})(1+\rho^N\mu+O(\rho^{N+1}))\in\Ga(\Cal
E[-2n])$. (Formally, the determinants are formed by using two
copies of the canonical section of $\La^nTM[-n]$ 
that expresses the
isomorphism between volume forms and 
densities of conformal weight $-n$ on an oriented manifold. Since two
copies of the forms are used, this is well defined even in the
non-orientable case, but this will not be needed here.)  To obtain
$\si$ and $\tau$, we have to take $\tfrac{-1}{2n}$th powers, and
taking into account that $\si$ is $\Cal O(\rho)$, we get
\begin{equation}\label{tau-si}
  \tau-\si=-\si\tfrac{\rho^N}{2n}\mu+\Cal O(\rho^{N+2}), 
\end{equation}
so this is $\Cal O(\rho^{N+1})$. Moreover, contracting \eqref{h-g2} with $g^{ai}$, we
get $g^{ai}h_{ij}=\delta^a_j+\rho^{N-2}g^{ai}\mu_{ij}$, which in turn easily
implies that
\begin{equation}\label{hinv}
  h^{ij}=g^{ij}-\rho^{N-2}g^{ik}\mu_{k\ell}g^{\ell j}+\Cal O(\rho^{N+3}).
\end{equation}
Hence $h^{ij}-g^{ij}$ is $\Cal O(\rho^{N+2})$, which in particular shows that, as
claimed above, $h$ is ALH provided that $g$ has this property.

\subsection{Tractors}\label{2.4}
For a manifold $K$ of dimension $n\geq 3$ which is endowed with a conformal
structure, the \textit{standard tractor bundle} \cite{BEG,CG-TAMS} is a vector bundle
$\Cal TK=\Cal E^A\to K$ of rank $n+2$ endowed with the following data.
\begin{itemize}
\item A Lorentzian bundle metric, called the \textit{tractor metric}, which we
  denote by $\langle\ ,\ \rangle$.
  \item A distinguished isotropic line subbundle $\Cal T^1K\subset \Cal TK$ that is
    isomorphic to the density bundle $\Cal E[-1]$.
\item A canonical linear connection, called the \textit{tractor connection}, that
  preserves the tractor metric and satisfies a non-degeneracy condition.
\end{itemize}
Since $\Cal T^1K$ is isotropic it is contained in $(\Cal T^1K)^\perp$, and
$\langle\ ,\ \rangle$ induces a positive definite bundle metric on $(\Cal
T^1K)^\perp/\Cal T^1K$. Via the tractor connection, this quotient gets identified
with $\Cal E^a[-1]$, so the tractor metric gives rise to a section of $\Cal
E_{ab}[2]$, which is exactly the \textit{conformal metric} $\mathbf{g}_{ab}$ that
defines the conformal structure on $K$. These properties together determine the data
uniquely up to isomorphism. Observe that the inclusion of $\Cal T^1K\cong\Cal E[-1]$
into $\Cal TK=\Cal E^A$ can be viewed as defining a canonical section
$\mathbf{X}^A\in\Ga(\Cal E^A[1])$. Moreover, pairing with $\mathbf{X}^A$, with
respect to the tractor metric, defines an isomorphism $\Cal TK/(\Cal
T^1K)^\perp\to\Cal E[1]$.


Our strategy for this article requires usage of tractors in a slightly
  unusual setting. A choice of metric in an equivalence class $\Cal G$ as in \ref{2.3}
  gives us a conformal structure on $\barm$ which therefore determines a tractor
  bundle $\Cal T\barm$. Each of these conformal structures restricts to the same
  conformal structure $[\Cal G_\infty]$ on $\partial M$, which is thereby canonically
  associated to $\Cal G$. Correspondingly, we have a tractor bundle $\Cal T\partial
  M$ which is canonically associated to $\Cal G$. While there is no easy way to
  relate the tractor bundles on $\barm$ coming from different choices of metrics in
  $\Cal G$, we can explicitly identify $\Cal T\partial M$ with a subbundle of $\Cal
  T\barm$ for any of the choices. Our basic method will be to construct tractor object on
  $M$ (making choices), prove that they admit a smooth extension to the boundary and
  that their boundary values lie in the subbundle $\Cal T\partial M$. Since this
  bundle is canonical we can then compare the results obtained from different
  choices.   

For our current purposes the ``naive'' approach to tractors (which avoids the
explicit use of Cartan connections or similar tools) is most appropriate and we'll
describe this next. A crucial feature in all approaches to tractors is that the
standard tractor bundle admits a simple description depending on the choice of a
metric $g$ in the conformal class. Such a choice gives an isomorphism $\Cal E^A\cong
\Cal E[1]\oplus\Cal E_a[1]\oplus \Cal E[-1]$, with the last summand corresponding to
$\Cal T^1K$ and the last two summands corresponding to $(\Cal T^1K)^\perp$. The
resulting elements are usually written as column vectors, with the first
component in the top, and there are simple explicit formulae for the tractor metric
and the tractor connection in these terms, namely:
\begin{equation}\label{trac-met}
\left\langle\begin{pmatrix}\si \\ \mu_a\\ \nu
\end{pmatrix},\begin{pmatrix}\tilde\si \\ \tilde\mu_a\\ \tilde\nu
\end{pmatrix}\right\rangle=\si\tilde\nu+\nu\tilde\si+\mathbf{g}^{ab}\mu_a\tilde\mu_b
\end{equation}
with $\mathbf{g}^{ab}$ denoting the inverse of the conformal metric, and
\begin{equation}\label{trac-conn}
\qquad \nabla^{\Cal T}_a \begin{pmatrix}\si\\ \mu_b\\ \nu\end{pmatrix}=
    \begin{pmatrix}\nabla_a\si-\mu_a\\ \nabla_a\mu_b+\mathbf{g}_{ab}\nu+
        \Rho_{ab}\si\\
        \nabla_a\nu-\mathbf{g}^{ij}\Rho_{ai}\mu_j\end{pmatrix}.
\end{equation}
In the right-hand side of this, we use the Levi-Civita connection and the Schouten
tensor $\Rho_{ab}$ of $g$. This is a trace-modification of the Ricci tensor
$\text{Ric}_{ab}$ of $g$ characterized by $\text{Ric}_{ab}=(n-2)\Rho_{ab}+\Rho
\mathbf{g}_{ab}$, where $\Rho : =\mathbf{g}^{ij}\Rho_{ij}$.

Changing from $g$ to a metric $\widehat{g}=e^{2f}g$ for
$f\in C^\infty(K,\Bbb R)$ there is an explicit formula for the change of the
identification in terms of $\Up_a=df$, namely
\begin{equation}\label{trac-transf}
\begin{pmatrix}\widehat{\si}\\ \widehat{\mu_a}\\ \widehat{\nu}
    \end{pmatrix}=\begin{pmatrix} \si\\ \mu_a+\Up_a\si
  \\ \nu-\mathbf{g}^{ij}(\Up_i\mu_j+\frac{1}{2}\Up_i\Up_j\si) \end{pmatrix} .
\end{equation}
Now one may turn around the line of argument and define sections of the tractor
bundle as equivalence classes of quadruples consisting of a metric in the conformal
class and sections of $\Cal E[1]$, $\Cal E_a[1]$ and $\Cal E[-1]$ with respect to the
equivalence relation defined in \eqref{trac-transf}. Recall that the behavior of the
Schouten tensor under a conformal change is given by
\begin{equation}\label{Rho-transf}
  \widehat{\Rho}_{ab}=\Rho_{ab}-\nabla_a\Up_b+\Up_a\Up_b-\tfrac12\Up_i\Up^ig_{ab}. 
\end{equation}
Using this, direct computations show that the definitions in \eqref{trac-met} and
\eqref{trac-conn} are independent of the choice of metric, so they give rise to a
well-defined bundle metric and linear connection on the resulting bundle. That this
cannot be done in individual points is a consequence of the fact that
tractors are more complicated geometric objects than tensors, since the action of
conformal isometries in a point depends on the two-jet of the isometry in that
point. To come to a point-wise construction, one would have to use 1-jets of metrics
in the conformal class instead.

\medskip

In the above discussion we have assumed that $n\geq 3$. Indeed, it is well known that
conformal structures in dimension two behave quite differently from higher
dimensions. In particular, they do not allow an equivalent description in term of a
normal Cartan geometry or of tractors. Still we can obtain boundary tractors as
follows. Note first, that associating to a conformal structure a tractor bundle and a
tractor metric via formulae \eqref{trac-transf} and \eqref{trac-met} works without
problems in dimension two. This observation is already sufficient for most of our
results, where we just need a vector bundle canonically associated to some structure
on the boundary. Now one view into the different behavior in dimension two is seen by
the fact that the definition of the Schouten tensor $\Rho_{ab}$ via the Ricci
curvature breaks down. While there are other ways to understand the Schouten tensor
it is nevertheless true that on a 2 dimensional Riemannian manifold there is no
natural tensor that transforms conformally according to (\ref{Rho-transf}).  Thus one
cannot use \eqref{trac-conn} to associate to a conformal structure a canonical
connection on the tractor bundle.

However, in the computations needed to verify that \eqref{trac-conn}
leads to a well-defined connection only the transformation law
\eqref{Rho-transf} for the Schouten tensor under conformal changes is
needed, the relation to the Riemann curvature does not play a
role. (In fact this computation only involves single covariant
derivatives, so there is no chance for curvature terms to arise.)
Consequently, the construction of a canonical connection on the
tractor bundle extends to dimension two, provided that in addition to
a conformal class one associates to each metric in that class a
symmetric tensor $\Rho_{ab}$, such that the tensors associated to
conformally related metrics satisfy the transformation law
\eqref{Rho-transf}.

The observation just made, for constructing a tractor bundle in
dimension 2, is close to the idea of a M\"obius structure, but
actually it is a slight generalization of the concept of a M\"obius
structure in the sense of \cite{Calderbank:Moebius}; compare in
particular with the MR review \cite{Eastwood:Moebius} of that
article. To define a M\"obius structure, one requires, in addition,
that the trace of the tensor $\Rho_{ab}$ associated to $g_{ab}$ is one
half times the scalar curvature of $g_{ab}$. This can be expressed as
a normalization condition on the curvature of the tractor connection,
but we will not need this. However, in the cases in which we will need
the tractor connection in dimension two, we actually will deal with
M\"obius structures, since we get flat tractor connections.

\subsection{Boundary values of tractors}\label{2.5}

In our usual setting of $\barm=M\cup \partial M$, equipped with the conformal class
defined by a conformally compact metric, we will deal with standard-tractor-valued
differential forms on $\barm$. The strategy is to associate to suitable forms, of
this type, a boundary value, and interpret this as a form taking values in the
standard tractor bundle $\Cal T\partial M$ of the conformal infinity. As discussed at
the end of Section \ref{2.4} above, such a bundle is available in all relevant
dimensions $\dim(\barm)=n\geq 3$. Moreover, since this infinity is the same for all
the metrics in a class $\Cal G$, as discussed in Section \ref{2.2}, this allows us to
relate the boundary values obtained from different metrics in $\Cal G$. However, even
for $n\geq 4$, it is not obvious how to relate $\Cal T\barm$ and $\Cal T\partial M$,
so we discuss this next. This discussion will also show how to canonically obtain the
``abstract Schouten tensors'' needed to define a tractor connection on $\Cal
T\partial M$ for $n=3$. Hence we obtain a uniform description of boundary values in
all dimensions.

Provided that one works with metrics that are smooth up to the boundary, the
discussion of boundary values of tractors can be reduced to the case of hypersurfaces
in conformal manifolds. Observe first that in our setting, it is no problem to relate
density bundles on $\barm$ and on $\partial M$ of any conformal weight. This is based
on the fact that $\Cal E[2]$ can always be viewed as the line subbundle spanned by
the conformal class and one can form boundary values for metrics in the conformal
class (that are smooth up to the boundary). So the densities of weight $w$ on
$\partial M$ are simply the restriction of the ambient densities of weight $w$
i.e.\ sections of $\mathcal{E}[w]|_{\partial M}$.

Now the conformal metric and its inverse define inner products on
$\Cal E^a[-1]$ and its dual $\Cal E_a[1]$. In the case of a boundary,
there thus is a unique inward pointing unit normal
$n^i\in\Ga(T\barm[-1]|_{\partial M})$ to the subbundle $T\partial M$,
and we put $n_i:=\mathbf{g}_{ij}n^j\in \Ga(T^*\barm[1]|_{\partial
  M})$. We will assume that $n^i$ and $n_i$ are (arbitrarily) extended
off the boundary, if needed. For a choice of metric $\bar{g}$ (which
is smooth up to the boundary) in the conformal class, we observe that
the restriction of $\nabla^{\bar g}_in_j$ to $T\partial M\x T\partial
M$ is independent of the chosen extension. This is the (weighted)
second fundamental form of $\partial M$ in $\barm$ with respect to
$\bar g$, and we can decompose it into a trace-free part and a
trace-part with respect to $\mathbf{g}_{ab}$. It is well-know that the
trace-free part of the second fundamental form is conformally
invariant.  The trace part can be encoded into the mean curvature
$H^{\bar g}\in \Gamma(\mathcal{E} [-1])$ of $\partial M$ in $\barm$
with respect to $\bar g$. (In our conventions $H^{\bar
  g}=\frac{1}{n-1}(\nabla_i n^i-n^in^j\nabla_in_j)$.)  Its behavior
under a conformal change corresponding to $\Up_a$ is given by
$H^{\widehat{\bar g}}=H^{\bar g}+\Up_in^i$. It is a classical fact,
see Section 2.7 of \cite{BEG}, that these ingredients can be
used to construct a conformally invariant {\em normal
    tractor}. This provides the standard way to determine whether
   the boundary
  value of an interior tractor is a boundary tractor. See also
  \cite{Gover:P-E} for the proof of the last part.
 
  \begin{prop}\label{prop2.5}
Let $\barm=M\cup\partial M$ be a smooth manifold with boundary that is endowed with a
conformal structure (which is smooth up to the boundary). Then there is a canonical
normal tractor $N^A\in\Ga(\Cal TM|_{\partial M})$. If $\bar g$ is a metric in the
class that is smooth up to the boundary then, in the splitting corresponding to $\bar
g$, $N^A=(0,n_i,-H^{\bar g})$, so $N^AX_A=0$ and for the tractor metric $h_{AB}$ we
get $h_{AB}N^AN^B=1$. For $n\geq 4$, the tractor bundle $\Cal T\partial M$ with
respect to the restriction of the conformal class can be canonically identified with
the orthocomplement ${N^A}^\perp\subset \Cal T\barm|_{\partial M}$.
  \end{prop} 
We next discuss how to make the last statement explicit following \cite{Gover:P-E}
and how this extends to the case $n=3$: In a scale $\bar g$ that is smooth up to the
boundary, a triple $(\si,\mu_a,\nu)$ is orthogonal to $N^A$ if and only if
$\mu_jn^j=H^{\bar g}\si $. Such a triple then gets mapped to
$$(\si,\mu_a-H^{\bar g} n_a\si,\nu+\tfrac12(H^{\bar g})^2\si)$$ in the splitting
corresponding to $\bar g|_{\partial M}$, see Section 6.1 of \cite{Curry-Gover}. In
particular, in the case that $\partial M$ is minimal with respect to $\bar g$
(i.e.\ that $H^{\bar g}$ vanishes identically), we simply get the na\"{\i}ve
identification of triples.  Given this identification one can compare, along
$\partial M$, the ambient tractor connection with that intrinsic to the conformal
structure of the boundary in dimensions $n\geq 4$. Although we will not need this
fact here, we note that, in particular, the boundary intrinsic tractor connection
agrees with the pull-back of the ambient tractor connection if the boundary is
totally umbilic and an object called the Fialkow tensor vanishes. Both hold if the
structure is locally asymptotically Einstein to a sufficient (low) order. See
\cite[Theorem 7.4]{Curry-Gover}, and its proof.

As mentioned above, these considerations also show how to obtain a
tractor connection on $\Cal T\partial M$ in the case that $n=3$. We
can do this in the setting of hypersurfaces and as discussed in
Section \ref{2.4}, we have to associate an ``abstract Schouten
tensor'' to the metrics in the conformal class $\partial M$.  The idea
here is simply that for metrics $\bar g$ such that $H^{\bar g}=0$, we
associate the restriction of the Schouten tensor of $\bar g$ to
$\partial M$ as an ``abstract Schouten tensor'' for the metric $\bar
g|_{\partial M}$. If $\bar g$ and $\widehat{\bar g}$ are two such
metrics, then for the change $\Up_a$, we get $\Up_in^i=0$, which
implies that the restriction of $\Up_i\Up^i$ to $\partial M$ coincides
with the $\bar{g}$-trace of the restriction of $\Up_a\Up_b$ to $T\partial M$. From
this and the Gauss formula we conclude that restrictions of the
Schouten tensors to $\partial M$ satisfy the correct transformation
law \eqref{Rho-transf}. This is already sufficient to obtain a tractor
connection on $\Cal T\partial M$. Since we know the behavior of all
objects under conformal rescalings, one can deduce a description of
$\Rho_{ab}$ for general metrics, but we won't need that here. A
different approach to induced M\"obius structures on hypersurfaces and
more general submanifolds in conformal manifolds can be found in
\cite{Belgun}.

In any case, it is clear from this description that the above
discussion of boundary values now extends to $n=3$. Finally, consider
a class $\Cal G$, as discussed in Section \ref{2.3}, with $N\geq 3$ and metrics
$g,h\in\Cal G$. Then for the conformal classes $[g]$ and $[h]$ the
metrics $\rho^2g$ and $\rho^2h$, for a local defining function $\rho$,
admit a smooth extension to the boundary.  Now by definition
$\rho^2g\sim_N\rho^2h$. So the difference of
their curvatures is $\Cal O(\rho^{N-2})$ and $N-2\geq 1$, which
implies that the restrictions of their Schouten tensors to the
boundary agree. This implies that, in dimension $n=3$, all metrics in
$\Cal G$ lead to the same tractor connection on $\Cal T\partial M$. In
higher dimensions the tractor connection on $\Cal T\partial M$ is
determined by the conformal structure $[\Cal G_\infty]$, and so the
equivalent result holds trivially.

\subsection{The scale tractor}\label{2.6}
We now combine the ideas about boundary tractors with the conformally
compact situation. This needs one more basic tool of tractor calculus,
the so-called tractor $D$-operator (also called the Thomas
$D$-operator). In the simplest situation, which is all that we need
here, this is an operator $D^A:\Cal E[w]\to\Cal E^A[w-1]$, which in
triples is defined by
\begin{equation}\label{D-def}
D^A\tau=\left(w(n+2w-2)\tau, (n+2w-2)\nabla_a\tau,
-\mathbf{g}^{ij}(\nabla_i\nabla_j+\Rho_{ij})\tau\right). 
\end{equation}
Again, a direct computation shows that this is conformally invariant. We will mainly
need the case that $w=1$, so that $D^A$ maps sections of the quotient bundle $\Cal
E[1]$ of $\Cal E^A$ to sections of $\Cal E^A$. Since the first component of
$\tfrac1nD^A\si$ is $\si$, this operator is referred to as a \textit{splitting
  operator}. In particular, given a metric in a conformal class, we can apply this
splitting operator to the canonical section of $\Cal E[1]$ obtained from the volume
density of the metric. The resulting section of $\Cal E^A$ is called the
\textit{scale tractor} associated to the metric.

Computing in the splitting determined by the metric itself, the associated section
$\si$ of $\Cal E[1]$ satisfies $\nabla_a\si=0$. Hence, in this splitting,
$\tfrac1nD^A\si$ corresponds to $(\si,0,-\tfrac1n\Rho\si)$, where recall
$\Rho=\mathbf{g}^{ij}\Rho_{ij}$. Applying the tractor connection to this, we get
$(0,(\Rho_{ab}-\frac{1}{n}\textbf{g}_{ab}\Rho)\si,-\frac{1}{n}\si\nabla_c\Rho)$. Observe
that the middle slot of this vanishes iff $\Rho_{ab}$ is pure trace, i.e.\ iff the
metric is Einstein. In that case, $\Rho$ which is just a multiple of the scalar
curvature, is constant, hence a metric is Einstein iff its scale tractor is parallel.

\medskip

Now we move to the case of $\barm=M\cup\partial M$ and a conformally
compact metric $g$ on $M$. By Proposition \ref{prop2.1}, the canonical
section $\si\in\Ga(\Cal E[1])$ which is parallel for $\nabla^g$ admits
a smooth extension to the boundary (as a defining density). Since we
have a conformal structure on all of $\barm$ also the scale tractor
$I^A:=\tfrac{1}{n}D^A\si$ and its covariant derivative $\nabla^{\Cal T}_aI^A$
are smooth up to the boundary. On $M$, we can compute in the splitting
determined by $g$, which shows that $\langle
I,I\rangle=-\tfrac2n\si^2\Rho=-\tfrac2ng^{ij}\Rho_{ij}$. Now
the scalar curvature $R=g^{ij}\text{Ric}_{ij}$ of $g$ can
be written as $2(n-1)\Rho$ and thus $\langle
I,I\rangle=-\tfrac1{n(n-1)}R$. In particular, if $g$ is ALH, then this
is identically $1$ along the boundary.

Under slightly stronger assumptions, the restriction of $I^A$ to the
boundary coincides with the normal tractor $N^A$ from Proposition
\ref{prop2.5}. Therefore, $I^A$ can be used to recognize objects that
lie in the boundary tractor bundle.

\begin{prop}\label{prop2.6}
  Consider a manifold $\barm$ with boundary $\partial M$ and interior $M$ and a
  conformally compact metric $g$ on $M$; let $\si\in\Ga(\Cal E[1])$ be the
  corresponding density and $I^A:=\tfrac1nD^A\si$ the scale tractor.

If $\langle I,I\rangle=1+\Cal O(\rho^2)$ near to $\partial M$, then the restriction
of $I^A$ to $\partial M$ coincides with the normal tractor $N^A$.
%
\end{prop}
\begin{proof}
 See Proposition 7.1 of \cite{Curry-Gover} or Proposition 6 of \cite{G-sigma}).
%
\end{proof}

\section{The tractor mass cocycle}\label{3}

We consider the tractor version of the classical asymptotically hyperbolic
mass here, so the order of asymptotics we need corresponds to
$N=n=\dim(M)$ in the notation of Section \ref{2}.

\subsection{The contribution from the trace}\label{3.1}
Most of the theory we develop applies in the general setting of an oriented manifold
$\barm$ with boundary $\partial M$, interior $M$, and equipped with a class $\Cal G$
of metrics on $M$, as introduced in Section \ref{2.3} with $N=n$. The only additional
assumption is that the metrics in $\Cal G$ are ALH in the sense of Definition
\ref{def2.2}. Given two metrics $g,h\in\Cal G$, we denote by $\si,\tau\in \Ga(\Cal
E[1])$ the corresponding powers of the volume densities of $g$ and $h$. Recall that
the class $\Cal G$ gives rise to a well-defined standard tractor bundle $\Cal
T\partial M$ over $\partial M$. Our aim is to associate to $g$ and $h$ a form
$c(g,h)\in\Om^{n-1}(\partial M,\Cal T\partial M)$, i.e.~a top-degree form on
$\partial M$ with values in $\Cal T\partial M$. Further, we want this to satisfy a
cocycle property, i.e.~that $c(h,g)=-c(g,h)$ and that $c(g,k)=c(g,h)+c(h,k)$ for
$g,h,k\in\Cal G$.

The first ingredient for this is rather simple: Given $g,h\in\Cal G$,
we use the conformal structure $[g]$ on $\barm$, and consider the
$\Cal T\barm$-valued one-form
$\tfrac{1}n\nabla^{\Cal T}_bD^A(\tau-\si)$. We already know that
$\tau$ and $\si$ admit a smooth extension to the boundary, so this is
well defined and smooth up to the boundary. Now on $M$, we can apply
the Hodge-$*$-operator determined by $g$ to convert this into a
$\Cal T\barm$-valued $(n-1)$-form. The following result shows that
this is smooth up to the boundary and that its boundary value is
orthogonal to the normal tractor $N^A$ (and non-zero in general). By
Proposition \ref{prop2.5}, this boundary value hence defines a form in
$\Om^{n-1}(\partial M,\Cal T\partial M)$.

\begin{prop}\label{prop3.1}
In the setting $\barm=M\cup\partial M$ and $g,h\in\Cal G$ as described above, let
$\rho$ be a local defining function for the boundary. Put $\overline{g}:=\rho^2g$, 
let $\overline{\si}\in\Ga(\Cal E[1])$ be the corresponding density and let
$\overline{g}_{\infty}$ be the boundary value of $\overline{g}$.

(1) In terms of the canonical section $\textbf{X}^A\in\Cal T\barm[1]$ from Section
\ref{2.4} and the function $\mu$ from \eqref{tau-si} and writing $\rho_a$ for
$d\rho$, we get
\begin{equation}\label{c11}
\nabla^{\Cal T}_bD^A(\tau-\si)=
\tfrac{n^2-1}{2}\rho^{n-2}\rho_b\mu\overline{\si}^{-1}\mathbf{X}^A+ \Cal
O(\rho^{n-1}).
\end{equation}

(2) The form $\star_g\nabla^{\Cal T}_bD^A(\tau-\si)$ is smooth up to the boundary and
its boundary value is given by
\begin{equation}\label{c12}
\tfrac{n^2-1}{2}\vol_{\overline{g}_{\infty}}\mu_\infty\overline{\si}_{\infty}^{-1}\mathbf{X}^A.
\end{equation}
Here $\vol_{\overline{g}_{\infty}}$ is the volume form of $\overline{g}_\infty$,
$\overline{\si}_{\infty}$ is the corresponding $1$-density on $\partial M$, and
$\mu_{\infty}$ is the boundary value of $\mu$. In particular, this is perpendicular
to $N^A|_{\partial M}$ and thus defines an $(n-1)$-form with values in $\Cal T\partial
M$.
\end{prop}
\begin{proof}
Observe first that for a connection $\nabla$ and a section $s$ that
both are smooth up to the boundary, and an integer $k>0$, we get
$\nabla_a(\rho^k s)=k\rho^{k-1}\rho_as+\Cal O(\rho^k)$. This can be
applied both to the tractor connection $\nabla^{\Cal T}$ and to the Levi-Civita
connection $\nabla$ of $\overline{g}$. 

Using that $\overline{\si}=\tfrac{\si}{\rho}$, we can write formula \eqref{tau-si}
(still for $N=n$) as $\tau-\si=-\overline{\si}\tfrac{\rho^{n+1}}{2n}\mu+\Cal
O(\rho^{n+2})$. Now the defining formula \eqref{D-def} for $D^A$ shows that the first
two slots of $D^A(\tau-\si)$ in the splitting determined by $\overline{g}$ are $\Cal
O(\rho^{n+1})$ and $\Cal O(\rho^n)$, respectively, while in the last slot the only
contribution which is not $\Cal O(\rho^n)$ comes from the double derivative. This
shows that
$$
D^A(\tau-\si)=-\tfrac{(n+1)}{2}\rho^{n-1}\overline{\si}\mu(-
\textbf{g}^{ij}\rho_i\rho_j)\mathbf{X}^A+\Cal O(\rho^n). 
$$
Next, using that $\overline{g}^{ij}=\overline{\si}^2\textbf{g}^{ij}$ and that $g$ is
ALH, we conclude that $\overline{\si}
\textbf{g}^{ij}\rho_i\rho_j=\overline{\si}^{-1}+\Cal O(\rho)$, so we get
$$
D^A(\tau-\si)=\tfrac{n+1}{2}\rho^{n-1}\mu\overline{\si}^{-1}\mathbf{X}^A
+\Cal O(\rho^n). 
$$
From this \eqref{c11} and hence part (1) follows immediately.

(2) Since $\barm$ is oriented we have an isomorphism
$ \Cal E[-n]\stackrel{\cong}{\longrightarrow} \La^nT^*\barm $
that can be interpreted as a canonical section $\bep_{a_1\dots a_n}\in
\Gamma (\La^nT^*\barm [n])$. In terms of this, the volume form of $g$ is given by
$\si^{-n}\bep_{a_1\dots a_n}$. Now by definition,
$\star_g\rho^{n-2}\rho_a$ is given by contracting $\rho^{n-2}\rho_a$ into the volume
form of $g$ via $g^{-1}$. So this is given by
\begin{equation}\label{sigma-tech}
  \si^{-n}\rho^{n-2}g^{ij}\rho_i\bep_{ja_1\dots a_{n-1}}.
\end{equation}
Now $\si^{-2}g^{ij}=\textbf{g}^{ij}$, while
$\si^{2-n}\rho^{n-2}=\overline{\si}^{2-n}$. Together with part (1), this shows that
$$
\star_g\nabla^{\Cal T}_aD^A(\tau-\si)=\tfrac{n^2-1}2\overline{\si}^{2-n}\mathbf{g}^{ij}\rho_i\bep_{ja_1\dots
  a_{n-1}}\mu\overline{\si}^{-1}\mathbf{X}^A+O(\rho).
$$
This is evidently smooth up to the boundary and its boundary value is a multiple of
$\mathbf{X}^A$ and thus perpendicular to $N^A$ by Proposition
  \ref{prop2.5}. To obtain the interpretation of the boundary value, we can rewrite
\eqref{sigma-tech} as $\overline{\si}^{-n}\overline{g}^{ij}\rho_i\bep_{ja_1\dots
  a_{n-1}}$. Since the first and last terms combine to give the volume form of
$\overline{g}$ and, along the boundary, $\overline{g}^{ib}\rho_i$ gives the unit
normal with respect to $\overline{g}$, we conclude that the boundary value of
\eqref{sigma-tech} is the volume form of $\overline{g}_\infty$. From this,
\eqref{c12} and thus part (2) follows immediately.
\end{proof}

There actually is a simpler way to write out the boundary value of
$\star_g\nabla^{\Cal T}_bD^A(\tau-\si)$ than \eqref{c12} that needs less choices. The function
$\mu$ defined in \eqref{tau-si} of course depends on the choice of the defining
function $\rho$, and there is no canonical choice of defining function. However,
fixing the metric $g\in\Cal G$, we of course get the distinguished defining density
$\si$, and we can get a more natural version of \eqref{tau-si} by phrasing things
in terms of densities. Namely, for the current setting with $N=n$, we can define
$\nu\in\Ga(\Cal E[-n])$ to be the unique density such that
\begin{equation}\label{tau-si-dens}
  \tau-\si=-\tfrac{1}{2n}\si^{n+1}\nu+\Cal O(\rho^{n+2}). 
\end{equation}
Then of course $\nu$ is uniquely determined by $g$ and $h$. In terms of a defining
function $\rho$ and the corresponding function $\mu$, we get
$\nu=(\tfrac{\rho}{\si})^n\mu$, which shows that $\nu$ is smooth up to the boundary
and non-zero wherever $\mu$ is non-zero. Let us write the boundary value of $\nu$ as
$\nu_\infty$, which by Section \ref{2.5} can be interpreted as a density of weight
$-n$ on $\partial M$.  In the setting of Proposition \ref{prop3.1}, we then have
$\overline{g}=\rho^2g$, so $\overline{\si}=\tfrac{\si}{\rho}$. The latter is smooth
up to the boundary and from Section \ref{2.5} we know that its boundary value is the
$1$-density $\overline{\si}_\infty$ on $\partial M$ corresponding to
$\overline{g}_\infty$. Hence our construction implies that $\vol_{\overline{g}_\infty}$
corresponds to the density $\overline\si_{\infty}^{1-n}$, so \eqref{c12} 
  for the boundary value of $\star_g\nabla^{\Cal T}_bD^A(\tau-\si)$ simplifies
 to
  \begin{equation}\label{c_1-density}
\tfrac{n^2-1}2\overline{\si}_\infty^{-n}\mu_\infty\mathbf{X}^A=
\tfrac{n^2-1}2\nu_\infty\mathbf{X}^A
  \end{equation}
Using this we can easily prove that we have constructed a cocycle.

\begin{cor}\label{cor3.1}
Let us denote by $c_1(g,h)$ the section of $\Cal T\partial M[-n+1]$ associated to
$g,h\in\Cal G$ via formula \eqref{c12}. Then $c_1$ is a cocycle in the sense that
$c_1(h,g)=-c_1(g,h)$ and that $c_1(g,k)=c_1(g,h)+c_1(h,k)$ for $g,h,k\in\Cal G$.
\end{cor}
\begin{proof}
  Expanding $\tau-\si=-\tfrac{1}{2n}\si^{n+1}\nu+\Cal O(\rho^{n+2})$
  as in \eqref{tau-si-dens} we can compute $c_1(g,h)$ via formula
  \eqref{tau-si-dens} from the boundary value of $\nu$. But
  \eqref{tau-si-dens} implies $\tau=\si+\Cal O(\rho^{n+1})$ and hence
  $\si-\tau=-\tfrac{1}{2n}\tau^{n+1}(-\nu)+\Cal O(\rho^{n+2})$ and
  hence $c_1(h,g)=-c_1(g,h)$. The second claim follows similarly.
\end{proof}

\begin{remark}\label{rem3.1} The usual classification results for
    invariant differential operators apply only to irreducible bundles, i.e.~natural
    vector bundles induced by irreducible representations of the conformal group and
    not to tractor bundles. However, we can use these results to prove that
    $D^A:\Ga(\Cal E[1])\to\Ga(\Cal TK)$ and $\nabla_a^{\Cal T}D^A:\Ga(\Cal
    E[1])\to\Om^1(K,\Cal TK)$ are the unique conformally invariant differential
    operators with the given source and target. It is known that the only invariant
    differential operator of low order defined on $\Ga(\Cal E[1])$ with values in an
    irreducible bundle that is non-zero on conformally flat manifolds is the
    ``conformal--to--Einstein operator'' that has values in $S^2_0T^*K[1]$. We'll
    discuss this operator in more detail in Section \ref{3.9}.

    Now each quotient of two subsequent components of the filtration $$\Cal
    T^1K\subset(\Cal T^1K)^\perp\subset\Cal TK$$ from Section \ref{2.4} splits into a
    direct sum of irreducible bundles and $S^2_0T^*K[1]$ is not among these bundles. Now
    given an invariant differential operator $\Ga(\Cal E[1])\to\Ga(\Cal TK)$,
    projecting back to sections of $\Cal TK/(\Cal T^1K)^\perp\cong\Cal E[1]$ one has
    to obtain a multiple of the identity. Hence subtracting an appropriate multiple
    of $D^A$, one obtains an operator $\Ga(\Cal E[1])\to\Ga((\Cal T^1K)^\perp)$. Now
    the projection of this to $\Ga((\Cal T^1K)^\perp/\Cal T^1K)$ has to vanish,
    whence the values actually have to lie in $\Ga(\Cal T^1K)$. This has to vanish
    since $\Cal T^1K$ is irreducible and not isomorphic to $\Cal E[1]$ or
    $S^2_0T^*K[1]$.

    Similar arguments apply to $\nabla^{\Cal T}_aD^A$ using the induced filtration of
    $T^*K\otimes\Cal TK$, which again leads to quotients that are direct sums of
    irreducible bundles.  This time $S^2_0T^*K[1]\subset T^*K\otimes (\Cal
    T^1K)^\perp/\Cal T^1K$ is among these bundles, but for all the other bundles
    obtained in this way it is easy to also see that there are no possible invariant
    operators coming from $\Ga(\Cal E[1])$ since there is not enough room for
    curvature terms.

  In particular, any invariant differential operator $\Gamma(\Cal
  E[1])\to\Om^1(K,\Cal TK)$ has to have values in $\Ga(T^*K\otimes(\Cal T^1K)^\perp)$
  and projecting to $\Ga(T^*K\otimes (\Cal T^1K)^\perp/\Cal T^1K)$ has to lead to a
  multiple of the conformal--to--Einstein operator. In particular, $\nabla^{\Cal
    T}_aD^A$ induces the conformal--to--Einstein operator in this way. Now if we have
  given an invariant differential operator $\Gamma(\Cal E[1])\to\Om^1(K,\Cal TK)$ we
  can subtract an appropriate multiple of $\nabla^{\Cal T}_aD^A$ to obtain an
  invariant operator $\Gamma(\Cal E[1])\to \Gamma(T^*K\otimes\Cal T^1K)$, which has
  to vanish by the above considerations.
  \end{remark}

\subsection{The contribution from the trace-free part}\label{3.2}
We next need another element of tractor calculus that, again, concerns
one-forms with values in the standard tractor bundle. Returning to the
setting of a general conformal manifold $K$, for $k=1,\dots,n$, there
is a natural bundle map $\partial^*:\La^kT^*K\otimes \Cal
TK\to\La^{k-1}T^*K\otimes\Cal TK$, which is traditionally called the
\textit{Kostant codifferential}. This has the crucial feature that
$\partial^*\o\partial^*=0$, so for each $k$, one obtains nested
natural subbundles $\im(\partial^*)\subset\ker(\partial^*)$ and hence
there is the subquotient $\Cal
H_k=\ker(\partial^*)/\im(\partial^*)$. To make this explicit in low
degrees let us write the spaces $\Lambda^kT^*K\otimes\Cal TK$ for
$k=0,1,2$ in the obvious extension of the vector notation for standard
tractors:
$$
\begin{pmatrix} \Cal E[1] \\ \Cal E_c[1] \\ \Cal E[-1]\end{pmatrix}
\overset{\partial^*}{\longleftarrow}
\begin{pmatrix} \Cal E_a[1] \\ \Cal E_{ac}[1] \\ \Cal E_a[-1]\end{pmatrix}
\overset{\partial^*}{\longleftarrow}
\begin{pmatrix} \Cal E_{[ab]}[1] \\ \Cal E_{[ab]c}[1] \\ \Cal E_{[ab]}[-1]\end{pmatrix} .
$$
Here $\Cal E_{[ab]}$ is the abstract index notation for $\La^2T^*K$.

The operators $\partial^*$ are discussed in \cite{book}, but what we need here
follows from some simple facts and observations, as follows. The maps $\partial^*$
are conformally invariant bundle maps, so they are induced by linear maps between
representations of $CO(n)$ which are equivariant for the action of the
group. Equivariancy under dilations implies that $\partial^*$ maps each row to the
row below, so in particular the bottom row is contained in the kernel of
$\partial^*$. Moreover, Kostant's version of the Bott-Borel-Weil Theorem implies that
$\Cal H_0\cong \Cal E[1]$ and $\Cal H_1\cong \Cal E_{(ab)_0}[1]$ (symmetric
trace-free part). In particular, $\partial^*:T^*K\otimes\Cal TK\to\Cal
  TK$ has to map onto the two bottom rows, so $\im(\partial^*)=(\Cal T^1K)^\perp$.
Hence $\Cal H_0$ coincides with the natural quotient bundle $\Cal E[1]$ of $\Cal TK$
considered above. This also implies that $\partial^*$ maps the top slot
  of $T^*K\otimes\Cal TK$ isomorphically onto the middle slot of $\Cal TK$, while
its restriction to the middle slot must be a non-zero multiple of the trace. Hence
$\ker(\partial^*)\subset T^*K\otimes\Cal TK$ consists exactly of those
elements for which the top slot vanishes and the middle slot is trace-free.

  Similarly, we conclude that $\partial^*:\La^2T^*K\otimes\Cal TK\to
    T^*K\otimes\Cal TK$ has to map the top slot injectively into the middle slot
  and the middle slot onto the bottom slot of $T^*K\otimes\Cal TK$.
  This shows how $\Cal H_1\cong \Cal E_{(ab)_0}[1]$ naturally arises as the
  subquotient $\ker(\partial^*)/\im(\partial^*)$ of $T^*K\otimes\Cal TK$.

Now similarly to the tractor-D operator, the machinery of BGG sequences constructs a
conformally invariant \textit{splitting operator}
$$
S:\Ga(\Cal H_1K)\to \Ga(\ker(\partial^*))\subset \Om^1(K,\Cal TK).
$$ Apart from the fact that for the projection $\pi_H:\ker(\partial^*)\to\Cal H_1$
one obtains $\pi_H(S(\al))=\al$, this operator is characterized by the single
property that, for the covariant exterior derivative $d^{\nabla^{\Cal T}}$ induced by
the tractor connection, one gets $\partial^*\o d^{\nabla^{\Cal T}}(S(\al))=0$ for any
$\al\in\Ga(\Cal H_1)$ \cite{CSS-annals}. Similar (but much easier)
  arguments as in Remark \ref{rem3.1} show that, up to constant multiples, $S$ is the
  unique invariant differential operator mapping $S^2_0T^*K$ to $\Om^1(K,\Cal TK)$.

To compute the explicit expression for $S$, we again use the notation of triples.
\begin{lemma}\label{lem3.2}
  Let $K$ be a conformal manifold and let $g$ be a metric in the conformal
  class. Then for $\ph=\ph_{ab}\in\Ga(\Cal E_{(ab)_0}[1])$, the section
  $S(\ph)\in\Om^1(K,\Cal TK)$ is, in the splitting determined by $g$, given by
  \begin{equation}\label{S-formula}
        (0,\ph_{ab}, \tfrac{-1}{n-1}\mathbf{g}^{ij}\nabla_i\ph_{aj}). 
  \end{equation}
\end{lemma}
\begin{proof}
From above, we know that $S$ has values in $\ker(\partial^*)$ and that this implies
that the first component of $S(\ph)$ has to be zero. The fact that $\pi_H\o S$ is the
identity map then shows that the middle component has to coincide with
$\ph_{ab}$. Thus it remains to determine the last component, which we temporarily
denote by $\ps=\ps_a\in\Ga(\Cal E_a[-1])$. This can be determined by exploiting the
fact that $\partial^*\o d^{\nabla^{\Cal T}}(S(\ph))=0$. To compute $d^{\nabla^{\Cal T}}(0,\ph_{bc},\ps_b)$,
we first have to use formula \eqref{trac-conn} to compute
$\nabla^{\Cal T}_a(0,\ph_{bc},\ps_b)$ viewing the form index $b$ as a mere ``passenger
index''. This leads to $(-\ph_{ba},\nabla_a\ph_{bc}+\mathbf{g}_{ac}\ps_b,*)$ where we
don't compute the last component, which will not be needed in what follows. Then we
have to apply twice the alternation in $a$ and $b$, which kills the first component
by symmetry of $\ph$ and leads to $2(\nabla_{[a}\ph_{b]c}-\ps_{[a}\mathbf{g}_{b]c})$
in the middle component. From the description of $\partial^*$ above we know that
$\partial^*\o d^{\nabla^{\Cal T}}(S(\ph))=0$ is equivalent to the fact that this middle
component lies in the kernel of a surjective natural bundle map to $\Cal E_a[-1]$.
By naturality, this map has to be a nonzero multiple of the contraction by
$\mathbf{g}^{bc}$. Using trace-freeness of $\ph$, we conclude that $\partial^*\o
d^{\nabla^{\Cal T}}(S(\ph))=0$ is equivalent to
$$
0=\mathbf{g}^{ij}(-\nabla_i\ph_{aj})-(n-1)\ps_a, 
$$
which gives the claimed formula. 
\end{proof}

Now we return to our setting $\barm=M\cup\partial M$, and a class
$\Cal G$ of metrics with $N=n$ as before. Given two metrics
$g,h\in\Cal G$, we now consider the trace-free part
$(h_{ij}-g_{ij})^0$ of $h_{ij}-g_{ij}$ with respect to $g$, which
defines a smooth section of $\Cal E_{(ab)_0}$. Thus for the density
$\si\in\Ga(\Cal E[1])$ determined by $g$, we can apply the splitting
operator $S$ to $\si(h_{ij}-g_{ij})^0$, to obtain a $\Cal
T\barm$-valued one-form $\ph_a^B$. We next prove that this has the
right asymptotic behavior to apply $\star_g$ and construct a boundary
value which lies in $\Om^{n-1}(\partial M,\Cal T\partial M)$, as we
did for the trace part in Section \ref{3.1} above.

Choosing a local defining function $\rho$ for the boundary, we get the function
$\mu_{ij}$ defined in \eqref{h-g}, and then 
\begin{equation}\label{h-g0}
  (h_{ij}-g_{ij})^0=\rho^{n-2}(\mu_{ij}-\tfrac1n\rho^2\mu g_{ij})+\Cal O(\rho^{n-1}), 
\end{equation}
and clearly $\mu^0_{ij}:=\mu_{ij}-\tfrac1n\rho^2\mu g_{ij}$ defines a section of
$\Cal E_{(ab)_0}$ that is smooth up to the boundary. 

\begin{prop}\label{prop3.2}
  In the setting and notation of Proposition \ref{prop3.1} and using $\mu_{ij}^0$ as
  defined above, we get:

  (1) The form $S(\si(h_{ij}-g_{ij})^0)\in\Cal E_a^A$ is given by
  \begin{equation}\label{c21}
     -\rho^{n-2}\overline{g}^{ij}\rho_i\mu^0_{aj}\overline{\si}^{-1}\mathbf{X}^A+\Cal
     O(\rho^{n-1}).
  \end{equation}

  (2) The form $\star_gS(\si(h_{ij}-g_{ij})^0)$ is smooth up to the boundary and its
  boundary value is given by
  \begin{equation}\label{c22}
     -\overline{g}^{ij}\overline{g}^{k\ell}\rho_i\rho_k\mu^0_{j\ell}\vol_{\overline{g}_\infty}
     \overline{\si}_{\infty}^{-1}\mathbf{X}^A.
  \end{equation}
  This is perpendicular to $N^A|_{\partial M}$ and thus by Proposition
    \ref{prop2.5} defines a form in $\Om^{n-1}(\partial M,\Cal T\partial M)$.
\end{prop}
\begin{proof}
(1) As before, we will work in the splitting determined by $\overline{g}_{ij}$
  throughout the proof. By construction
  $\si(h_{ij}-g_{ij})^0=\si\rho^{n-2}\mu_{ij}^0=\overline{\si}\rho^{n-1}\mu_{ij}^0$
  is $\Cal O(\rho^{n-1})$. Using Lemma \ref{lem3.2}, we see that the first slot of
  $S(\si(h_{ij}-g_{ij})^0)$ vanishes and its middle slot is $\Cal
  O(\rho^{n-1})$. Using the observation from the beginning of the proof of
  Proposition \ref{3.1} we see that the covariant derivative of
  $\overline{\si}\rho^{n-1}\mu_{ij}^0$, with respect to the Levi-Civita connection of $\overline{g}_{ij}$, is given
  by $(n-1)\overline{\si}\rho^{n-2}\rho_k\mu_{ij}^0+\Cal O(\rho^{n-1})$.  Using this,
  the claimed formula follows immediately from Lemma \ref{lem3.2} and the fact that
  $\overline{g}^{ij}=\overline{\si}^2\mathbf{g}^{ij}$.

  (2) Proceeding as in the proof of Proposition \ref{prop3.1}, we now show that
\mbox{$\star_gS((h_{ij}-g_{ij})^0)$} is given by
  $$
 -\overline{g}^{ij}\rho_i\mu^0_{kj}\overline{g}^{k\ell}\bep_{\ell a_1\dots
   a_{n-1}}\overline{\si}^{-n-1}\mathbf{X}^A. 
 $$

 This is evidently smooth up to the boundary and, as observed there,
 $\overline{\si}^{-n}\bep_{a_1\dots a_n}$ is the volume form of
 $\overline{g}_{ij}$. Writing $\vol_{\overline{g}}|_{\partial M}$ as
 $d\rho\wedge\vol_{\overline{g}_\infty}$ and using that the image of $d\rho$ in
 $\Om^1(\partial M)$ vanishes, we directly get the claimed formula for the boundary
 value. The final statement follows the same argument as in Proposition \ref{prop3.1}.
\end{proof}

Similarly as in Section \ref{3.1} above, this admits a more natural interpretation
when working with densities. Again fixing $N=n$, instead of \eqref{h-g} we can start
from
\begin{equation}\label{h-g-dens}
  h_{ij}=g_{ij}+\si^{n-2}\nu_{ij}, 
\end{equation}
where $\nu_{ij}\in\Ga(\Cal E_{(ij)}[-n+2])$ now is a weighted symmetric two-tensor
that is smooth up to the boundary. For a choice of local defining function $\rho$,
the relation to \eqref{h-g} is described by
$\nu_{ij}=(\frac{\rho}{\si})^{n-2}\mu_{ij}$. This immediately implies that
$$
\mathbf{g}^{ij}\nu_{ij}=(\tfrac{\rho}{\si})^{n-2}\tfrac1{\si^2}g^{ij}\mu_{ij}=(\tfrac{\rho}{\si})^n\mu=\nu,
$$
where $\nu\in\Ga(\Cal E[-n])$ is the density used in Section
\ref{3.1}.  The tracefree part $\nu^0_{ij}$, then of course is $\nu_{ij}-\tfrac1n
\mathbf{g}_{ij}\nu=(\frac{\rho}{\si})^{n-2}\mu^0_{ij}$. On the other hand, the fact
that $g_{ij}$ is ALH shows that $\tfrac{1}{\rho^2}g^{ij}\rho_i\rho_j$ is identically
one along the boundary. Using $g^{ij}=\si^2\mathbf{g}^{ij}$, we conclude that
$\tfrac{\si}{\rho}\mathbf{g}^{ij}\rho_i\in\Cal E^a[-1]$ coincides, along $\partial M$,
with the conformal unit normal $n^j$ from Section \ref{2.5}. Hence
$\overline{g}^{ij}\rho_j=\overline{\si}n^i$ and hence we can rewrite formula
\eqref{c21} as $-\si^{n-2}\nu^0_{ij}n^j\mathbf{X}^A+\Cal O(\rho^{n-1})$, where we
have extended $n^j$ arbitrarily off the boundary.

To rewrite the formula \eqref{c22} for the boundary value in a similar
way, we use the observation that $\vol_{\overline{g}_\infty}$
corresponds to $(\tfrac{\rho}{\si})^{n-1}$, as discussed in Section
\ref{3.1}. Using this and the above, see that \eqref{c22} equals
      \begin{equation}\label{tfp-dens}
-n^in^j(\nu^0_{ij})_\infty\mathbf{X}^A, 
\end{equation}
where $(\nu^0_{ij})_\infty$ indicates the boundary value of $\nu^0_{ij}$. Using this
formulation, it is easy to prove that we obtain another cocycle.

\begin{cor}\label{cor3.2}
Let us denote the section of $\Cal T\partial M[-n+1]$ associated to $g,h\in\Cal G$
via formula \eqref{c22} by $c_2(g,h)$. Then $c_2$ is a cocycle in the sense that
$c_2(h,g)=-c_2(g,h)$ and that $c_2(g,k)=c_2(g,h)+c_2(h,k)$ for $g,h,k\in\Cal G$.
\end{cor}
\begin{proof}
  Suppose that the tracefree part of $h_{ij}-g_{ij}$ with respect to $g$ is given by
  $\si^{n-2}\nu^0_{ij}$ for $\nu^0_{ij}\in\Ga(\Cal E_{ij}[2-n])$ and similarly the
  tracefree part of $(g_{ij}-h_{ij})^0$ with respect to $h$ corresponds to
  $\tilde\nu^0_{ij}$. Then one immediately verifies that
  $\tilde\nu^0_{ij}=-\nu^0_{ij}+\Cal O(\rho^{n-2})$, and thus using formula
    \eqref{tfp-dens} to compute boundary values readily implies
    $c_2(h,g)=-c_2(g,h)$. The second claim follows similarly. 
\end{proof}

\subsection{Diffeomorphisms}\label{3.3}
We next start to study the compatibility of the cocycles we have
constructed above with diffeomorphisms. There are various concepts of
compatibility here, and we have to discuss some background
first. Recall that a diffeomorphism $\Psi:\barm\to\barm$ of a manifold
with boundary maps $M$ to $M$ and $\partial M$ to $\partial M$. This
also shows that for $x\in\partial M$, the linear isomorphism
$T_x\Psi:T_xM\to T_{\Psi(x)}M$ maps the subspace $T_x\partial M$ to
$T_{\Psi(x)}\partial M$. It follows that for a defining
function $\rho$ for $\partial M$, also $\Psi^*\rho=\rho\o\Psi$ is a
defining function for $\partial M$. This also works for a local
defining function defined on a neighborhood of $\Psi(x)$, which gives
rise to a local defining function defined on a neighborhood of $x$.

Using these results implies that the relation $\sim_N$ on tensor
fields defined in Section \ref{2.3} is compatible with diffeomorphisms
in the sense that $t\sim_N\tilde t$ implies $\Psi^*t\sim_N\Psi^*\tilde
t$ for each $N>0$. In particular, given an equivalence class $\Cal G$
of conformally compact metrics, also $\Psi^*\Cal G$ is such an
equivalence class. We are particularly interested in the case that
$\Psi^*\Cal G=\Cal G$, in which we say that \textit{$\Psi$ preserves
  $\Cal G$}. The diffeomorphisms with this property clearly form a
subgroup of the diffeomorphism group $\Diff(\barm)$ which we denote by
$\Diff_{\Cal G}(\barm)$. From our considerations it follows
immediately that this is equivalent to the fact that there is one
metric $g\in\Cal G$ such that $\Psi^*g\in\Cal G$ or equivalently
$\Psi^*g\sim_{N-2}g$. (In \cite{CDG} an analogous property is phrased
by saying that $\Psi$ is an ``asymptotic isometry'' of $g$. We don't
use this terminology since $\Psi$ is not more compatible with $g$
than with any other metric in $\Cal G$.)

Recall from Section \ref{2.3} that all metrics in $\Cal G$ give rise to the same
conformal infinity on $\partial M$. This implies that for $\Psi\in\Diff_{\Cal
  G}(\barm)$ the restriction $\Psi_\infty:=\Psi|_{\partial M}$ is not only a
diffeomorphism, but actually a conformal isometry of the conformal infinity of $\Cal
G$. In particular, it induces a well-defined bundle automorphism on the standard
tractor bundle $\Cal T\partial M$ and hence we can pull back sections of $\Cal
T\partial M$ along $\Psi_\infty$. This also works for $n=3$ without
problems. Using this, we can prove the first and
simpler compatibility condition of our cocycles with diffeomorphisms.

\begin{prop}\label{prop3.3}
Consider a manifold $\barm=M\cup\partial M$ with boundary and a class $\Cal G$ of
metrics, in the case $N=n$. Then the cocycles constructed in Propositions
\ref{prop3.1} and \ref{prop3.2} are compatible with the action of a diffeomorphism
$\Psi\in\Diff_{\Cal G}(\barm)$ in the sense that for each such cocycle
$c(\Psi^*g,\Psi^*h)=(\Psi_\infty)^*c(g,h)$. Here $\Psi_\infty=\Psi|_{\partial M}$,
and on the right hand side we have the action of a conformal isometry of the
conformal infinity of $\Cal G$ (on $\partial M$) on tractor-valued forms.
\end{prop}
\begin{proof}
This basically is a direct consequence of the invariance properties of the
constructions we use. If $g$ corresponds to $\si\in\Ga(\Cal E[1])$, then of course
$\Psi^*g$ corresponds to $\Psi^*\si$. Moreover, $\Psi_\infty$ defines a conformal
isometry between the conformal structures on $\partial M$ induced by $\Psi^*g$ and
$g$, respectively. Similarly, $\Psi^*h$ corresponds to $\Psi^*\tau$ and naturality of
the tractor constructions
implies that
$\nabla^{\Cal T}_aD_B(\Psi^*\tau-\Psi^*\si)$ (computed in the conformal structure $[\Psi^*g]$)
equals $\Psi^*(\nabla^{\Cal T}_aD_B(\tau-\si))$. Since $\Psi|_M$ is an isometry from $\Psi^*g$
to $g$, we get $\Psi^*\vol_g=\vol_{\Psi^*g}$, which implies compatibility with the
Hodge-star. Hence on $M$, we get
$$
\star_{\Psi^*g}\nabla^{\Cal T}_aD_B(\Psi^*\tau-\Psi^*\si)=\Psi^*(\star_g\nabla^{\Cal T}_aD_B(\tau-\si))
$$
and since both sides admit a smooth extension to the boundary, the boundary values
have to coincide, too. But these than are exactly $c_1(\Psi^*g,\Psi^*h)$ and the
pull back induced by the conformal isometry $\Psi_\infty$ of $c_1(g,h)$. This
completes the proof for $c_1$.

For $c_2$, we readily get that $(\Psi^*h-\Psi^*g)^0$ (tracefree part with respect to
$\Psi^*g$) coincides with $\Psi^*(h-g)^0$ (tracefree part with respect to $g$). Using
naturality of the splitting operator $S$, the proof is completed in the
same way as for $c_1$. 
\end{proof}

\subsection{Diffeomorphisms asymptotic to the identity}\label{3.4}
To move towards a more subtle form of compatibility of our cocycles with
diffeomorphisms, we need a concept of asymptotic relation between diffeomorphisms.

\begin{definition}\label{def3.4}
Let $\barm=M\cup\partial M$ be a manifold with boundary and let
$\Ps,\tilde\Ps:\barm\to\barm$ be diffeomorphisms.

(1) We say that $\Ps$ and $\tPsi$ are asymptotic of order $N>0$ and write
$\Psi\sim_N\tPsi$ if and only if for any function $f\in C^\infty(\barm, \Bbb R)$
we get $f\o\Psi\sim_Nf\o\tPsi$ in the sense of Section \ref{2.3}.

(2) For $N>0$, we define $\Diff_0^N(\barm)$ to be the set of diffeomorphisms which are
asymptotic to the identity $\id_{\barm}$ of order $N$.
\end{definition}

Since $\sim_N$ clearly defines an equivalence relation on functions,
we readily see that it is an equivalence relation on
diffeomorphisms. Moreover, since the pull back of a local defining
function for $\partial M$ along a diffeomorphism of $\barm$ again is a
local defining function, we conclude that $\Psi\sim_N\tPsi$ implies
$\Psi\o\Ph\sim_N\tPsi\o\Ph$ and $\Ph\o\Psi\sim_N\Phi\o\tPsi$ for any
diffeomorphism $\Phi$ of $\barm$. In particular, this shows either of
$\tPsi^{-1}\o\Psi\sim_N\id$ and $\tPsi\o\Psi^{-1}\sim_N\id$ is
equivalent to $\Psi\sim_N\tPsi$.

On the other hand, we need some observations on charts. Given a manifold
$\barm=M\cup\partial M$ with boundary, take a point $x\in\partial M$. Then by
definition, there is a chart $(U,u)$ around $x$, so $U$ is an open neighborhood of
$x$ in $\barm$ and $u:U\to u(U)$ is a diffeomorphism onto an open subset of an
$n$-dimensional half space. Then $u$ restricts to a diffeomorphism between the open
neighborhood $U\cap\partial M$ of $x$ in $\partial M$ and the open subspace
$u(U)\cap\Bbb R^{n-1}\x\{0\}$ of $\Bbb R^{n-1}$. Then by definition, the last
coordinate function $u^n$ is a local defining function of $\partial M$. Conversely,
any local defining function can locally be used as such a coordinate function in a
chart.

If $\Psi\in\Diff(\barm)$ is a diffeomorphism, then for a chart $(U,u)$
also $(\Psi^{-1}(U),u\o\Psi)$ is a chart. If $\tPsi$ is another
diffeomorphism such that $\tPsi|_{\partial M}=\Psi|_{\partial M}$,
then $V:=\Psi^{-1}(U)\cap\tPsi^{-1}(U)$ is an open subset in $\barm$
which contains $\Psi^{-1}(U\cap\partial M)$. For any tensor field $t$
defined on $U$, both $\Psi^*t$ and $\tPsi^*t$ are defined on $V$, and
can be compared asymptotically there. Using these observations, we
start by proving a technical lemma.

\begin{lemma}\label{lem3.4}
  Let $\barm=M\cup\partial M$ be a smooth manifold with boundary, 
  let $\Psi,\tPsi\in\Diff(\barm)$ be diffeomorphisms, and fix
  $N>0$. Then the following conditions are equivalent:

  \begin{itemize}
  \item[(i)] $\Psi\sim_{N+1}\tPsi$
  \item[(ii)] $\Psi|_{\partial M}=\tPsi|_{\partial M}$ and for any tensor field $t$
    on $\barm$ (including functions), we get $\Psi^*t\sim_N\tPsi^*t$.
  \item[(iii)] $\Psi|_{\partial M}=\tPsi|_{\partial M}$ and for each $x\in\partial M$, there
    is an chart $(U,u)$ for $\barm$, with $x\in U$,  whose coordinate functions $u^i$
    satisfy $\Psi^*u^i\sim_{N+1}\tPsi^*u^i$ locally around $\Psi^{-1}(x)$. 
  \end{itemize}
\end{lemma}
\begin{proof}
  Replacing $\Psi$ by $\tPsi^{-1}\o\Psi$ we may without loss of generality assume
  that $\tPsi=\id_{\barm}$, which we do throughout the proof. 

  (i)$\Rightarrow$(iii): We first claim that $\Psi|_{\partial M}=\id_{\partial
    M}$. For $x\in\partial M$, take an open neighborhood $W$ of $x$ in $\partial
  M$. Then there is a bump function $f\in C^\infty(\barm,\Bbb R)$ with values in
  $[0,1]$ such that $f(x)=1$ and such that $f^{-1}(\{1\})\cap\partial M\subset W$. By
  assumption $f\o\Psi\sim_{N+1}f$, so in particular, these functions have to agree on
  $\partial M$ and hence at $x$. Since $\Psi(x)\in \partial M$, by construction, we
  get $\Psi(x)\in W$. Since $W$ was arbitrary, this implies that $\Psi(x)=x$ and
  hence the claim. Having this at hand, we take any chart $(U,u)$ with $x\in U$,
  extend the coordinate functions $u^i$ to globally defined functions on $M$ without
  changing them locally around $x$, and then (i) immediately implies that
  $\Psi^*u^i\sim_{N+1}u^i$ locally around $x$.

  \smallskip

  (iii)$\Rightarrow$(ii): For any tensor field $t$, it suffices to
  verify $\Psi^*t\sim_N t$ locally around each boundary point
  $x\in\partial M$. Fixing $x$, we take a chart $(U,u)$ as in (iii)
  and its coordinate functions $u^i$ and we work on
  $V=\Psi^{-1}(U)\cap U$. Taking a vector field $\xi\in\frak X(U)$ we
  can compare $\Psi^*\xi$ and $\xi$ on $V$. We can do this via
  coordinate expressions with respect to the chart $(U,u)$ and we
  denote by $\xi^i$ and $(\Psi^*\xi)^i$ the component functions. By
  assumption, $u^i\o\Psi=u^i+\Cal O(\rho^{N+1})$ and differentiating
  this with $\Psi^*\xi$, we obtain $(\Ps^*\xi) (u^i\o\Psi)=(\Ps^*\xi)(u^i)+\Cal
  O(\rho^N)$. Thus we conclude that
  $(\Psi^*\xi)(u^i\o\Psi)\sim_N(\Psi^*\xi)^i$. But by definition of
  the pull back, we get
  $(\Psi^*\xi)(u^i\o\Psi)=\xi(u^i)\o\Psi\sim_{N+1}\xi^i$. Overall, we
  conclude that $(\Psi^*\xi)^i\sim_N\xi^i$ on $V$, which implies that
  $\Psi^*\xi\sim_N\xi$ on $V$ and hence we get condition (ii) for
  vector fields.

  In particular, this implies that the coordinate vector fields $\partial_i$ of the
  chart $(U,u)$ satisfy $\Psi^*\partial_i\sim_N\partial_i$ on $V$. On the other hand,
  applying the exterior derivative to $u^i\o\Psi\sim_{N+1}u^i$, we conclude that
  $\Psi^*du^i=d(u^i\o\Psi)\sim_Ndu^i$. Of course, on $V$ the $du^i$ coincide with the
  coordinate one-forms of the chart $(U,u)$. Now given a tensor field $t$ of any
  type, we can take $\Psi^*t$ and hook in vector fields $\Psi^*\partial_{i_a}$ and
  one-forms $\Psi^*du^{j_b}$. On $V$ this by construction produces one of the
  component functions of $t$ up to $\Cal O(\rho^N)$. On the other hand, by definition
  of the pull back, this coincides with the composition of the corresponding coordinate
  function of $t$ with $\Psi$, and hence with that coordinate function up to $\Cal
  O(\rho^{N+1})$, so (ii) is satisfied in general.

  \smallskip

  (ii)$\Rightarrow$(i): Let $f\in C^\infty(\barm,\Bbb R)$ be a smooth function. Then
  by assumption we know that $f\o\Psi\sim_Nf$, so choosing a local defining function
  $\rho$ for $\partial M$, we get $f\o\Psi=f+\rho^N\tilde f$ for some smooth function
  $\tilde f\in C^\infty(\barm,\Bbb R)$. But then
  $\Psi^*df=d(f\o\Psi)=df+N\rho^{N-1}\tilde fd\rho+\Cal O(\rho^N)$. However,
  condition (ii) also says that $\Psi^*df\sim_Ndf$ and since $d\rho|_{\partial M}$ is
  nowhere vanishing, this implies that $\tilde f|_{\partial M}=0$. But this implies
  that $f\o\Psi\sim_{N+1}f$ and hence condition (i) follows. 
  \end{proof}

\subsection{The relation to adapted defining functions}\label{3.5}
In a special case and in quite different language, it has been observed in \cite{CDG}
that there is a close relation between diffeomorphisms asymptotic to the identity and
adapted defining functions. We start discussing this with the following lemma.

\begin{lemma}\label{lem3.5}
In our usual setting, of $\barm=M\cup\partial M$, let $\Cal G$ be an
equivalence class of ALH metrics on $M$ for the relation $\sim_{N-2}$
for some $N\geq 3$. Take two metrics $g,h\in\Cal G$, and let $\rho$
and $r$ be local defining functions for $\partial M$ defined on the
same open subset $U\subset\barm$. If $\rho$ is adapted to $g$ and $r$
is adapted to $h$ in the sense of Definition \ref{def2.2} and if
$\rho^2g_{ij}$ and $r^2h_{ij}$ induce the same metric on the boundary,
then $\rho\sim_{N+1}r$.
\end{lemma}
\begin{proof}
We have to analyze the asymptotics of solutions to the PDE that governs the change to
an adapted defining function.  Replacing $\barm$ by an appropriate open subset, we
may assume that $\rho$ and $r$ are defined on all of $\barm$. Then we can write
$r=\rho e^v$ for some smooth function $v\in C^\infty(\barm,\Bbb R)$ which gives
$dr=rdv+e^vd\rho$. In abstract index notation, this reads as
$r_i=rv_i+e^v\rho_i$. The fact that $r$ is adapted to $h_{ij}$ says that
$r^{-2}h^{ij}r_ir_j$ is identically $1$ on a neighborhood of the boundary, and
inserting we conclude that
\begin{equation}\label{adap-PDE}
1\equiv \rho^{-2}h^{ij}\rho_i\rho_j+2\rho^{-1}h^{ij}\rho_iv_j+h^{ij}v_iv_j. 
\end{equation}
Observe that $\rho^{-2}h^{ij}$, $\rho_i$ and $v_i$ are all smooth up to the boundary,
so the terms in the right hand side are $\Cal O(1)$, $\Cal O(\rho)$, and $\Cal
O(\rho^2)$, respectively.

Now on the one hand, since $g,h\in\Cal G$, we know from \eqref{hinv} that
$\rho^{-2}h^{ij}=\rho^{-2}g^{ij}+\Cal O(\rho^N)$. Since $\rho$ is adapted to $g$,
this means that the first term in the right hand side of \eqref{adap-PDE} is $1+\Cal
O(\rho^N)$. Inserting into \eqref{adap-PDE}, we conclude that
\begin{equation}\label{adap-PDE2}
2\rho^{-1}h^{ij}\rho_iv_j+h^{ij}v_iv_j=\Cal O(\rho^N).
\end{equation}

On the other hand, $r^2h_{ij}=e^{2v}\rho^2h_{ij}$, and $r^2h_{ij}$, by assumption, is
smooth up to the boundary with the same boundary value as $\rho^2g_{ij}$.  Hence our
assumption on the induced metrics on the boundary imply that $e^{2v}|_{\partial
  M}=1$, so $v$ has to vanish identically along the boundary and hence $v=\Cal
O(\rho)$. Inductively, putting $v=\rho^\ell\tilde v$ for $\ell\geq 1$, we get
$v_i=\ell\rho^{\ell-1}\tilde v\rho_i+\Cal O(\rho^\ell)$, which implies that the left
hand side of \eqref{adap-PDE2} becomes $2\ell\rho^\ell\tilde
v\rho^{-2}h^{ij}\rho_i\rho_j+\Cal O(\rho^{\ell+1})$. As long as $\ell<N$, this shows
that $\tilde v= \Cal O (\rho)$, so we conclude that we can write $v=\rho^N\tilde v$
where $\tilde v$ is smooth up to the boundary. But then
$$r=\rho e^v=\rho(1+\rho^N\tilde v+\Cal O(\rho^{N+1}))=\rho+\Cal O(\rho^{N+1}),$$
which completes the proof.
\end{proof}

Using this, we can now establish several important properties of
diffeomorphisms that are asymptotic to the identity.

\begin{thm}\label{thm3.5}
  Let $\barm=M\cup\partial M$ be a smooth manifold with boundary and, for some $N\geq
  3$, let $\Cal G$ be an equivalence class of ALH metrics on $M$ for the relation
  $\sim_{N-2}$. Let us denote by $[\Cal G_\infty]$ the conformal structure on $\partial M$
  defined by the conformal infinity of $\Cal G$ and by $\Conf(\partial M,[\Cal G_\infty])$
  the group of its conformal isometries. Then we have

(1) $\Diff^{N+1}_0(\barm)$ is a normal subgroup in $\Diff(\barm)$ and
  is contained in $\Diff_{\Cal G}(\barm)$.

(2) Restriction of diffeomorphisms to the boundary induces a homomorphism
  $\Diff_{\Cal G}(\barm)\to\Conf(\partial M,[\Cal G_\infty])$ with kernel
  $\Diff^{N+1}_0(\barm)$.
\end{thm}
\begin{proof}
(1) The observations on the relation $\sim_N$ for diffeomorphisms we have made after
  Definition \ref{def3.4} readily imply that $\Diff^{N+1}_0(\barm)$ is stable under
  inversions as well as compositions, and conjugations by arbitrary elements of
  $\Diff(\barm)$. Hence $\Diff^{N+1}_0(\barm)$ is a normal subgroup of
  $\Diff(\barm)$. Taking $g\in\Cal G$ and a local defining function $\rho$ for
  $\partial M$, we know that $\rho^2g_{ij}$ admits a smooth extension to the
  boundary. Thus, given $\Psi\in \Diff^{N+1}_0(\barm)$, we may apply part (ii) of
  Lemma \ref{lem3.4} to conclude that $\Psi^*(\rho^2g_{ij})\sim_N\rho^2g_{ij}$. Now
  $\Psi^*(\rho^2)=(\rho\o\Psi)^2$ and $\rho\o\Psi\sim_{N+1}\rho$. Hence
  $\rho^2\Psi^*g_{ij}\sim_N \Psi^*(\rho^2g_{ij})\sim_N\rho^2g_{ij}$ and restricting
  to $M$ we conclude that $\Psi^*g_{ij}\in\Cal G$. This completes the proof of (1).

  \smallskip
  
(2) It follows readily from the definitions that, for
  $\Psi\in\Diff_{\Cal G}(\barm)$, the diffeomorphism $\Psi|_{\partial
    M}$ of $\partial M$ preserves the conformal structure $[\Cal
    G_\infty]$. Thus we get a homomorphism $\Diff_{\Cal
    G}(\barm)\to\Conf(\partial M,[\Cal G_\infty])$ as claimed and it remains
  to prove the claim about the kernel. So let us take a diffeomorphism
  $\Psi\in\Diff_{\Cal G}(\barm)$ such that $\Psi|_{\partial
    M}=\id_{\partial M}$ and we want to show that
  $\Psi\sim_{N+1}\id$. To prove this, we can apply condition (iii) of
  Lemma \ref{lem3.4} and work locally around a boundary point $x$. Let
  us choose $g\in\Cal G$ and a local defining function $\rho$ for
  $\partial M$ which is adapted to $g$ and defined on some open
  neighborhood $U$ of $x$ in $\barm$. Now we consider the normal field
  (to $\rho $ level sets) determined by $g$ and $\rho$, i.e.~we put
  $\xi^i:=\rho^{-2}g^{ij}\rho_j$. This admits a smooth
  extension to the boundary,  and the fact that $\rho$ is adapted to $g$
  exactly says that $\rho^2g_{ij}\xi^i\xi^j$ is identically one on a
  neighborhood of the boundary.

  For $x\in\partial M$, we now work on an open neighborhood $W$ of $x$
  in $\barm$ such that $W\subset\Psi^{-1}(U)\cap U$. On $W$, we can
  pull back all our data by $\Psi$, thus obtaining
  $h_{ij}:=\Psi^*g_{ij}$, $r:=\rho\o\Psi$, and $\eta:=\Psi^*\xi$. By
  assumption, $h_{ij}\in\Cal G$ and pulling back the defining equation
  for $\xi^i$ we get $\eta^i=r^{-2}h^{ij}r_j$. Also by pulling back,
  we readily see that the $r^2h_{ij}\eta^i\eta^j$ is identically one
  on a neighborhood of the boundary. Hence we conclude that the local
  defining function $r$ is adapted to $h_{ij}$ and hence Lemma
  \ref{lem3.5} shows that $r\sim_{N+1}\rho$. Observe that this implies
  that $r_j\sim_N\rho_j$ and together with $g^{ij}\sim_{N+2}h^{ij}$
  the defining equations for $\xi$ and $\eta$ show that
  $\xi\sim_N\eta$.

  Next, we pass to an appropriate collar of the boundary. We can
  choose an open neighborhood $V$ of $x$ in $\partial M$ and a
  positive number $\epsilon$ such that the flow map
  $(y,t)\mapsto\Fl^\eta_t(y)$ defines a diffeomorphism $\ph$ from $V\x
  [0,\epsilon)$ onto an open subset contained in $\Psi^{-1}(W)\cap W$,
    and on which $\rho$ satisfies $(\rho^2 g)^{-1}(d\rho,d\rho)=1$.
    Now let us define $\partial_t:=\ph^*\eta$. Note that this is the
    coordinate vector field for any product chart on $V\x
    [0,\epsilon)$ induced by some chart on $V$. Since
      $\eta=\Psi^*\xi$, the fact that $\Psi$-related vector fields
      have $\Psi$-related flows together with $\Psi|_{\partial M}=\id$
      readily implies that $(\Psi\o\ph)(y,t)=\Fl^\xi_t(y)$. Using
      Section \ref{3.4}, and in particular Lemma \ref{lem3.4}, we see
      that we can complete the proof by showing that
      $\Psi\o\ph\sim_{N+1}\ph$ as follows.

  By Lemma \ref{lem3.4} it suffices to consider the pull backs of
  coordinate functions of local charts along these diffeomorphisms. We
  apply this to charts which are obtained by composing a product chart
  for $V\x [0,\epsilon)$ with $(\Psi\o\ph)^{-1}$. Now the fact that
    $d\rho(\xi)\equiv 1$ (and that $\Psi\o\ph$ maps $V\times \{0\}$
    into $\partial M$) shows that applying this construction to the
    coordinate $t$ on $[0,\epsilon)$, we obtain $\rho$. As observed
      above, $\rho\o\Psi\sim_{N+1}\rho$ and thus
      $\rho\o\Psi\o\ph\sim_{N+1}\rho\o\ph$. Thus it remains to
      consider functions $f$ such that $f\o(\Psi\o\ph)$ is one of the
      boundary coordinate functions, and hence $\partial_t\cdot (f\o
      \Ps\o\ph)\equiv 0$ or, equivalently, $\xi\cdot f\equiv 0$ on an
      appropriate neighborhood of the boundary. Note then that  $\ph^*\xi\cdot
      (f\o\phi)\equiv 0$. We have to compare
      $f\o(\Psi\o\ph)$ to $f\o\ph$.

      As we have observed above, we get $\eta\sim_N\xi$
            and hence
        $\ph^*\xi\sim_N\ph^*\eta=\partial_t$. Thus we get
        $\ph^*\xi=\partial_t+t^N\tilde\xi$ for some vector field
        $\tilde\xi\in\frak(V\x [0,\epsilon))$. By construction $f\o\ph$ and
          $f\o(\Ps\o\ph)$ agree on $V\x\{0\}$, so
          $f\o\ph\sim_1f\o(\Ps\o\ph)$. Assuming that $f\o\ph\sim_kf\o(\Ps\o\ph)$ for
          some $k\geq 1$, we get $f\o\ph=f\o(\Ps\o\ph)+t^k\tilde f$ for some $\tilde
          f\in C^\infty(V\x[0,\ep),\Bbb R)$. Then we compute
$$
0=(\partial_t+t^N\tilde\xi)\cdot (f\o(\Ps\o\ph)+t^k\tilde
f)=0+kt^{k-1}\tilde f+\Cal O(t^{\text{min}(k,N)}) .
  $$

  If $k\leq N$, then this equation shows that $\tilde f$ vanishes along $V\x\{0\}$
and hence $(f\o\ph)\sim_{k+1}f\o(\Ps\o\ph)$. Inductively, this gives
$(f\o\ph)\sim_{N+1}f\o(\Ps\o\ph)$, which completes the proof.
\end{proof}

\subsection{Aligned metrics}\label{3.6}
Following an idea in \cite{CDG} we next show that the freedom under diffeomorphisms
asymptotic to the identity can be absorbed in a geometric relation between the
metrics. The analogous condition in \cite{CDG} is phrased as ``transversality''. 

\begin{definition}\label{def3.6}
Let $\barm=M\cup\partial M$ be a manifold with boundary and let $\Cal G$ be an
equivalence class of ALH metrics for the relation $\sim_{N-2}$ for some $N\geq
3$. Consider two metrics $g,h\in\Cal G$ and a local defining function $\rho$ for
$\partial M$ defined on some open subset $U\subset\barm$. Then we say that $h$
\textit{is aligned with $g$ with respect to $\rho$} if
$$
\rho_ig^{ij}(h_{jk}-g_{jk})\equiv 0
$$
on some open neighborhood of $U\cap\partial M$ in $U$. 
\end{definition}

Observe that the condition in Definition \ref{def3.6} can be rewritten as
$\rho_ig^{ij}h_{jk}=\rho_k$. This in turn implies that $\rho_ih^{ij}=\rho_ig^{ij}$
and hence the gradients of $\rho$ with respect to the two metrics coincide
on a neighborhood of the boundary. This shows that in the case that $\rho$ is adapted
to $g_{ij}$ and $h_{ij}$ is aligned to $g_{ij}$ with respect to $\rho$, then $\rho$
is also adapted to $h_{ij}$.

\begin{thm}\label{thm3.6}
In our usual setting, of $\barm=M\cup\partial M$ and a class $\Cal G$
of metrics, assume that $g\in\Cal G$ and $\rho$ is a local defining
function for $\partial M$ that is adapted to $g$. Then for any
$h\in\Cal G$ and locally around any point $x\in\partial M$ in the
domain of definition of $\rho$, there exists a diffeomorphism $\Psi$
such that $\Psi\sim_{N+1}\id$ and such that $\Psi^*h$ is aligned to
$g$ with respect to $\rho$. Moreover, the germ of $\Psi$ along the
intersection of its domain of definition with $\partial M$ is uniquely
determined by this condition.
\end{thm}
\begin{proof}
Since $\rho$ is adapted to $g$ the function
$\rho^{-2}g^{ij}\rho_i\rho_j$ is identically one on some neighborhood
of the boundary. Now given $h_{ij}\in\Cal G$, we can use Proposition
\ref{prop2.2} to modify $\rho$ to a local defining function $r$
adapted to $h_{ij}$ in such a way that $\rho^2g_{ij}$ and $r^2h_{ij}$
induce the same metric on the boundary, compare also to the proof of
Lemma \ref{lem3.5}. Now we define vector fields $\xi=\xi^i$ and
$\eta=\eta^i$ by $\xi^i:=\rho^{-2}g^{ij}\rho_j$ and
$\eta^i:=r^{-2}h^{ij}r_j$ where as usual we write $\rho_j$ for $d\rho$
and similarly for $r$. Recall from the proof of Theorem \ref{thm3.5} that this implies
that $\xi\sim_N\eta$.


Now, via the construction of collars from the proof of Theorem \ref{thm3.5}, we
obtain a diffeomorphism $\Psi$ which has the property that for any $y$ in an
appropriate neighborhood of $x$ in $\partial M$, we get
$\Psi\o\Fl^\xi_t(y)=\Fl^\eta_t(y)$. Differentiating this equation shows that
$\xi=\Psi^*\eta$. By construction the derivative of the function $t\mapsto
r(\Fl^\eta_t(y))$ is given by
$$
dr(\eta)=r^{-2}r_ih^{ij}r_j\equiv 1, 
$$ so $r(y)=0$ shows that $r(\Fl^\eta_t(y))=t$. In the same way
$\rho(\Fl^\xi_t(y))=t$, which shows that $\Psi^*r=\rho$.
Knowing this and $\xi\sim_N\eta$, the last part of the proof of
Theorem \ref{thm3.5} shows that $\Psi\sim_{N+1}\id$.

To show that $\Psi$ has the required property, observe that by
construction $r^2h_{ij}\eta^j=r_i$. That is $i_\eta r^2h=dr$ and in the
same way $i_\xi \rho^2g=d\rho$.  Using this and the above, we now
obtain
$$
i_\xi\rho^2\Psi^*h=i_{\Psi^*\eta}\Psi^*r^2h=\Psi^*(i_\eta r^2h)=\Psi^*(dr)=d\rho=i_\xi
\rho^2g.  
$$
This shows that, if we insert $\xi$ into the bilinear form
$\rho^2(\Psi^*h-g)$ (which by construction is smooth up to the
boundary), the resulting one form vanishes identically on a neighborhood of the
boundary. This exactly shows that $\Psi^*h$ is aligned with $g$ with
respect to $\rho$, so the proof of existence is complete.

To prove uniqueness, assume that $h$ is aligned to $g$ with respect to $\rho$ and
that a diffeomorphism $\Psi$ such that $\Psi\sim_{N+1}\id$ has the property that also
$\Psi^*h$ is aligned to $g$ with respect to $\rho$. Observe that $\Psi\sim_{N+1}\id$
implies that $\Psi^*h\in\Cal G$. As observed above, the fact that both $h$ and
$\Psi^*h$ are aligned with $g$ with respect to $\rho$ implies that $\rho$ is adapted
both to $h$ and to $\Psi^*h$. But on the other hand, the fact that $\rho$ is adapted
to $h$ of course implies that $\Psi^*\rho$ is adapted to $\Psi^*h$. Now by
construction $\rho\o\Psi\sim_{N+1}\rho$ and hence $(\rho\o\Psi)\Psi^*h$ and
$\rho\Psi^*h$ induce the same metric on the boundary. Hence the uniqueness part in
Proposition \ref{prop2.2} implies that $\rho\o\Psi=\rho$ and hence
$\Psi^*\rho_i=\rho_i$ on a neighborhood of the boundary.

Now of course the inverse metric to $\Psi^*h_{ij}$ is $\Psi^*(h^{ij})$
and $\Psi^*\rho_i\Psi^*h^{ij}=\rho_i\Psi^*h^{ij}$. Since
$\Psi^*h_{ij}$ is aligned with $g_{ij}$ with respect to $\rho$, we get
$ \rho_i\Psi^*h^{ij}=\rho_ig^{ij}$ and since also $h_{ij}$ is aligned
with $g$ with respect to $\rho$, $ \rho_ig^{ij}$ equals $\rho_ih^{ij}$.
Hence putting $\xi^j:=\rho_ih^{ij}$ we conclude that $\Psi^*\xi=\xi$.
Since $\Psi$ is the identity on the boundary, this implies that
$\Psi(\Fl^\xi_t(y))=\Fl^\xi_t(y)$ for $y$ in the boundary and $t$
sufficiently small, so $\Psi=\id$ locally around the boundary.
\end{proof}

\subsection{The action of the aligning diffeomorphism}\label{3.7}
Fix $g\in\Cal G$ and a local defining function $\rho$ for the boundary which is
adapted to $g$. Consider another metric $h\in\Cal G$ and the corresponding tensor
field $\mu_{ij}$ defined by \eqref{h-g}. From Theorem \ref{thm3.6} we then know that
locally around each boundary point, we find an essentially unique diffeomorphism
$\Psi$ such that $\Psi\sim_{N+1}\id$ and such that $\Psi^*h$ is aligned with $g$ with
respect to $\rho$. Since $\Psi^*h\in\Cal G$, the analog of \eqref{h-g} defines a
tensor field $\tilde\mu_{ij}$ that describes the difference between $\Psi^*h$ and
$g$. We now prove that the boundary value of $\tilde\mu_{ij}$ can be explicitly
computed from the boundary value of $\mu_{ij}$. This will be the crucial step towards
finding combinations of the two cocycles constructed in Sections \ref{3.1} and
\ref{3.2} which are invariant under diffeomorphisms asymptotic to the
identity. Observe that this formally looks like the coordinate formula in Proposition
2.16 of \cite{CDG}, but the actual meaning is  different: Our description
does not involve any choice of local coordinates, but uses only abstract indices.

\begin{thm}\label{thm3.7}
In the setting of Theorem \ref{thm3.6} for some fixed order $N$, let
$\mu_{ij}$ be the tensor field relating $h_{ij}$ and $g_{ij}$
according to \eqref{h-g}, and let $\tilde\mu_{ij}$ be the
corresponding tensor field relating $\Psi^*(h_{ij})$ and
$g_{ij}$. Then putting $\xi^i:=\rho^{-2}g^{ij}\rho_j$, we obtain
\begin{equation}\label{alignment-diffeo}
  \tilde\mu_{ij}=\mu_{ij}-\rho_i\mu_{j\ell}\xi^\ell-\rho_j\mu_{i\ell}\xi^\ell+
  \tfrac{\xi^k\mu_{k\ell}\xi^\ell}{N}(\rho^2g_{ij}+(N-1)\rho_i\rho_j)+\Cal O(\rho).
\end{equation}
\end{thm}
\begin{proof}
We use the quantities introduced in the proof of Theorem \ref{thm3.6}:
We denote by $r$ the defining function adapted to $h_{ij}$ such that
$r^2h_{ij}$ and $\rho^2g_{ij}$ induce the same metric on the
boundary. Further we put $\xi^i:=\rho^{-2}g^{ij}\rho_j$ and
$\eta^i:=r^{-2}h^{ij}r_j$. In terms of these, we know from the proof
of Theorem \ref{thm3.6} that the alignment diffeomorphism $\Psi$
satisfies $r\o\Psi=\rho$ and $\Psi^*\eta=\xi$, and hence
$\Psi\o\Fl^\xi_t=\Fl^\eta_t$ wherever the flows are defined. Moreover,
writing $r=e^v\rho$, we know from the proof of Lemma \ref{lem3.5} that
$v=\rho^N\tilde v$ for a function $\tilde v$ that admits a smooth
extension to the boundary. Moreover, we can compute the boundary value
of $\tilde v$ from that proof: In equation \eqref{adap-PDE} we can
bring the first term on the right hand side to the left hand side and
then use the fact that $\rho$ is adapted to $g$ to rewrite
\eqref{adap-PDE} as
$$
(\rho^{-2}g^{ij}-\rho^{-2}h^{ij})\rho_i\rho_j=2\rho\rho^{-2}h^{ij}\rho_iv_j+\rho^2\rho^{-2}h^{ij}v_iv_j. 
$$ In the left hand side, we can insert \eqref{hinv} and use the definition of $\xi$
to obtain $\rho^N\xi^k\mu_{k\ell}\xi^\ell$. On the other hand $v_j=N\rho^{N-1}\tilde
v\rho_j+\Cal O(\rho^N)$. Inserting this in the right hand side and using that
$\rho^{-2}h^{ij}\rho_i\rho_j=1+\Cal O(\rho)$, we obtain $2N\rho^N\tilde v+\Cal
O(\rho^{N+1})$, which shows that
\begin{equation}\label{tildev}
\tilde v=\tfrac{1}{2N}\xi^k\mu_{k\ell}\xi^\ell+\Cal O(\rho). 
\end{equation}

The basis for the further computation will be the fact that for vector fields
$\ze_1,\ze_2$ (which we assume to be smooth up to the boundary), we get
$\Psi^*h(\Psi^*\ze_1,\Psi^*\ze_2)=h(\ze_1,\ze_2)\o\Psi$.  Multiplying by $\rho^2$, we
obtain
\begin{equation}\label{PB}
  \rho^2\Psi^*h(\Psi^*\ze_1,\Psi^*\ze_2)=(r^2h(\ze_1,\ze_2))\o\Psi,
\end{equation}
and both sides admit a smooth extension to the boundary. Hence the right hand side
equals $r^2h(\ze_1,\ze_2)+\Cal O(\rho^{N+1})$. Inserting $r=\rho e^v$ and
\eqref{h-g}, we conclude that this equals
$$
e^{2v}\rho^2g(\ze_1,\ze_2)+e^{2v}\rho^N\mu(\ze_1,\ze_2)+\Cal O(\rho^{N+1}).
$$
Of course, $e^{2v}=1+2\rho^N\tilde v+\Cal O(\rho^{N+1})$ and we conclude that the
right hand side of \eqref{PB} equals 
\begin{equation}\label{PB-RHS}
  \rho^2g(\ze_1,\ze_2)+\rho^N\big(2\tilde v\rho^2g(\ze_1,\ze_2)+
  \mu(\ze_1,\ze_2)\big)+\Cal O(\rho^{N+1}).
\end{equation}

The left hand side of \eqref{PB}, by definition, can be written as
$\rho^2g(\Psi^*\ze_1,\Psi^*\ze_2)+\rho^N\tilde\mu(\Psi^*\ze_1,\Psi^*\ze_2)$. Now we
know that $\Psi^*\ze_1\sim_N\ze_1$ and hence $\Psi^*\ze_1=\ze_1+\rho^N\tilde\ze_1$,
where $\tilde\ze_1$ admits a smooth extension to the boundary, and similarly for
$\ze_2$. Inserting this, we obtain
\begin{equation}\label{PB-LHS}
\rho^2g(\ze_1,\ze_2)+\rho^N\big(\rho^2g(\ze_1,\tilde\ze_2)+\rho^2g(\tilde\ze_1,\ze_2)+
\tilde\mu(\ze_1,\ze_2)\big)+\Cal O(\rho^{N+1}).
\end{equation}
Since this has to equal \eqref{PB-RHS}, we conclude that
\begin{equation}\label{tilde-mu-main}
  \tilde\mu(\ze_1,\ze_2)=\mu(\ze_1,\ze_2)-\rho^2g(\ze_1,\tilde\ze_2)-
  \rho^2g(\tilde\ze_1,\ze_2)+2\tilde v\rho^2g(\ze_1,\ze_2)+\Cal O(\rho). 
\end{equation}

The key observation now is that the computation of $\tilde\ze_1$ and $\tilde\ze_2$
essentially reduces to the computation of the vector field $\tilde\eta$ which has the
property that $\eta=\xi+\rho^N\tilde\eta$. As a first step, we claim that for a
vector field $\zeta$ which is tangent to the boundary along the boundary, we have
$\Psi^*\zeta\sim_{N+1}\zeta$. This can of course be proved locally, so we can use
local charts obtained from a collar construction as in the proof of Theorem
\ref{thm3.5}. These have $r$ as one coordinate and $\eta$ as the corresponding
coordinate vector field. We first consider the case that $\ze$ is the coordinate
vector field $\partial_i$ associated to one of the boundary coordinates. Of course,
$0=[\eta,\partial_i]$ and pulling back along $\Psi$, we conclude that
$0=[\xi,\Psi^*\partial_i]$. Now by Lemma \ref{lem3.4}, we know that $\xi\sim_N\eta$
and $\Psi^*\partial_i\sim_N\partial_i$, and we express this via
$\xi=\eta+r^N\tilde\eta$ and $\Psi^*\partial_i=\partial_i+r^N\tilde\zeta$, where
$\tilde\eta$ and $\tilde\ze$ admit a smooth extension to the boundary. Plugging these
expressions into the Lie bracket and using that $\partial_i\cdot r=0$ and $\eta\cdot
r=1$, we conclude that
$$
0=[\xi,\Psi^*\partial_i]=[\eta,\partial_i]+Nr^{N-1}\tilde\ze+\Cal O(r^N). 
$$
This shows that $\tilde\ze$ vanishes along the boundary and hence
$\Psi^*\partial_i\sim_{N+1}\partial_i$. Now a general vector field $\zeta$ that is
tangent to the boundary along the boundary can be be written as $f\eta+\sum
f_i\partial_i$ for arbitrary smooth functions $f_i$ and a smooth function $f$ which
is $\Cal O(r)$. Thus the general version of our claim follows readily since
$\Psi^*\eta\sim_N\eta$, $f_i\o\Psi\sim_{N+1}f_i$ and $f\o\Psi\sim_{N+1}f$. 

Now for any vector field $\ze$ that is smooth up to the boundary, the difference
$\ze-d\rho(\ze)\eta$ is smooth up to the boundary and tangent to the boundary along
the boundary. Of course $\Psi^*(d\rho(\ze)\eta)=(d\rho(\ze)\o\Psi)\xi$ and so this
equals $d\rho(\ze)\xi+\Cal O(\rho^{N+1})$. Thus,  writing
$\ze=d\rho(\ze)\eta+(\ze-d\rho(\ze)\eta)$ and pulling back, we get 
\begin{equation}\label{tilde-zeta}
\rho^N\tilde\ze=\Psi^*\ze-\ze=d\rho(\ze)(\xi-\eta)+\Cal O(\rho^{N+1}).
\end{equation}

To compute the difference $\xi-\eta$, we first
recall that $r_j=e^v\rho_j+e^v\rho v_j$ and $v_j=N\rho^{N-1}\tilde v\rho_j+\Cal O(\rho^N)$. This
shows that $r_j=\rho_j(1+(N+1)\rho^N\tilde v)+\Cal O(\rho^{N+1})$. Next, by
definition $\eta^i=e^{-2v}\rho^{-2}h^{ij}r_j$ and
$$
e^{-2v}(1+(N+1)\rho^N\tilde v)=1+(N-1)\rho^N\tilde v+\Cal O(\rho^{N+1}).
$$
Now \eqref{hinv} shows that
$$
\rho^{-2}h^{ij}=\rho^{-2}g^{ij}-\rho^N(\rho^{-2}g^{ik}\mu_{k\ell}\rho^{-2}g^{\ell
  j})+\Cal O(\rho^{N+1}). 
$$
Putting all this together, we see that
\begin{equation}\label{eta-xi}
\eta^j-\xi^j=\rho^N((N-1)\tilde v\xi^j-\rho^{-2}g^{jk}\mu_{k\ell}\xi^\ell)+\Cal
O(\rho^{N+1}).    
\end{equation}
Dividing the negative of the right hand side by $\rho^N$ and contracting with
$\rho^2g_{ij}$, we obtain $-(N-1)\tilde v\rho_j+\mu_{j\ell}\xi^{\ell}$. Using this
and \eqref{tilde-zeta}, we can write \eqref{tilde-mu-main} in abstract index notation 
as
\begin{equation}\label{tilde-mu-main2}
\tilde\mu_{ij}=\mu_{ij}+2(N-1)\tilde v\rho_i\rho_j-\rho_i\mu_{j\ell}\xi^\ell
-\rho_j\mu_{i\ell}\xi^\ell +2\tilde v\rho^2g_{ij}+\Cal O(\rho), 
\end{equation}
which together with \eqref{tildev} exactly gives the claimed formula. 
\end{proof}

Using this, we can easily deduce that appropriate combinations of the cocycles
constructed in Sections \ref{3.1} and \ref{3.2} remain unchanged if one of the two
metrics involved is pulled back by a diffeomorphisms that is asymptotic to the
identity. Since we are dealing with the situation of the classical mass here, we have
to specialize to the case that $N=n$. 

\begin{cor}\label{cor3.7}
In our usual setting of $\barm=M\cup\partial M$, let $\Cal G$ be an equivalence
class of ALH metrics on $M$ for the relation $\sim_{n-2}$. Let $c_1$ and $c_2$ be the
cocycles constructed in Sections \ref{3.1} and \ref{3.2}, respectively, and let $c:=\tfrac1nc_1+\tfrac12c_2$. Then $c$ defines a cocycle on
$\Cal G$ that has the property that for metrics $g,h\in\Cal G$ and any diffeomorphism
$\Phi\in \Diff_0^{n+1}(\barm)$, we get $c(g,h)=c(g,\Phi^*h)$.
\end{cor}
\begin{proof}
  We fix $g\in\Cal G$ and a local defining function $\rho$ that is adapted to
  $g$. Then we show that for $c=\tfrac1nc_1+\tfrac12c_2$ and the alignment
  diffeomorphism $\Psi$ obtained from Theorem \ref{thm3.7}, we get
  $c(g,\Psi^*h)=c(g,h)$. The last part of Theorem \ref{thm3.6} shows that $\Psi^*h$
  is the unique metric in the the orbit of $h$ under $\Diff_0^{n+1}(\barm)$ which is
  aligned to $g$ with respect to $\rho$. Hence applying the construction of Theorem
  \ref{thm3.7} to $\Phi^*h$ for arbitrary $\Phi\in \Diff_0^{n+1}(\barm)$, we also
  have to arrive at $\Psi^*h$, which then implies the result.

  Using the formulae in parts (2) of Propositions \ref{prop3.1} and \ref{prop3.2}, we
  see that to prove our claim it suffices to show that the boundary value of
  \begin{equation}\label{expr}
   \tfrac{n^2-1}{2n}\rho^{-2}g^{ij}\mu_{ij}-\tfrac12\xi^i\mu^0_{ij}\xi^j 
  \end{equation}
  coincides with the boundary value of the analogous expression formed from
  $\tilde\mu_{ij}$. Inserting
  $\mu^0_{ij}=\mu_{ij}-\tfrac{1}{n}\rho^{-2}g^{k\ell}\mu_{k\ell}\rho^2g_{ij}$ and
  using that $\xi^i\rho^2g_{ij}\xi^j=1$ on a neighborhood of $\partial M$, we see that
  \eqref{expr} equals 
  $$
   \tfrac{n}2\rho^{-2}g^{ij}\mu_{ij}-\tfrac12\xi^i\mu_{ij}\xi^j. 
   $$
   Contracting $\rho^{-2}g^{ij}$ into formula \eqref{alignment-diffeo} (for the case
   $N=n$) and multiplying by $\tfrac{n}2$, we obtain
   $$
   \tfrac{n}2\rho^{-2}g^{ij}\tilde\mu_{ij}=\tfrac{n}2\rho^{-2}g^{ij}\mu_{ij}-
   \tfrac12\xi^i\mu_{ij}\xi^j.  
    $$
   By alignment, $\xi^i\tilde\mu_{ij}=0$, so this proves our claim. 
\end{proof}

\begin{remark}\label{rem3.7} The computations in this section can also
    be used to show that the cocycle from Corollary \ref{cor3.7} simplifies in an
    important special case. Assume that we deal with two metrics $g_{ij}$ and
    $h_{ij}$ in $[\Cal G]$ such that there is a local defining function $\rho$ which
    is adapted to both metrics at the same time. By definition, this implies 
    $\rho_ih^{ij}\rho_j=\rho_ig^{ij}\rho_j$ on a neighborhood of the
    boundary. Assuming this, we can contract $\rho_i\rho_j$ into equation
    \eqref{hinv} and putting $\xi^i=\rho^{-2}g^{ij}\rho_j$ as above, the result reads
    as $0=\rho^{N+2}\xi^k\mu_{k\ell}\xi^\ell+\Cal O(\rho^{N+3})$. Thus we conclude
    that $ \xi^k\mu_{k\ell}\xi^\ell$ vanishes along the boundary. But in the proof of
    Corollary \ref{cor3.7} we have seen that $\xi^k\mu^0_{k\ell}\xi^\ell$ is a linear
    combination of $\xi^k\mu_{k\ell}\xi^{\ell}$ and of $\mu$. We have also seen there
    that our cocycle involves only the boundary values
    of $\xi^k\mu^0_{k\ell}\xi^\ell$ and of $\mu$, so under the current assumptions
    this reduces to a multiple of $\mu$ only.
  \end{remark}

\subsection{From relative to absolute invariants}\label{3.8}
So far, we have not imposed any restriction on the equivalence class
$\Cal G$ of metrics beyond the fact that it consists of
ALH-metrics. We next show that assuming that $\Cal G$ locally contains
metrics that are hyperbolic (i.e.~have constant sectional curvature
$-1$), one can use our construction to obtain an invariant for
(single) metrics in $\Cal G$. This assumption of course implies that
the conformal infinity $[\Cal G_\infty]$ on $\partial M$ is conformally flat,
but as we shall see below, it does not impose further restrictions on
the topology of $\barm$.

The key step toward this are results on the uniqueness of hyperbolic
metrics with prescribed infinity that are discussed in Chapter 7 of
\cite{FeffGr}. These build on results in \cite{SkenSol} and are
related to the work in \cite{Epstein}.

\begin{thm}\label{thm3.8}
Consider our usual setting, of $\barm=M\cup\partial M$, and an equivalence class
$\Cal G$ of ALH metrics on $M$ for the relation $\sim_{n-2}$. Assume that for each
$x\in\partial M$ there is an open neighborhood $U$ of $x$ in $\barm$ and a metric $g$
in $\Cal G$ that is hyperbolic (i.e.\ has constant sectional curvature $-1$) on
$U$. Let $c$ denote the cocycle  from
Corollary \ref{cor3.7}. 

Then for an open subset $U$ as  above, and two metrics $g_1,g_2\in\Cal
G$ that are hyperbolic on $U$, we get $c(g_1,h)|_{U\cap\partial
  M}=c(g_2,h)|_{U\cap\partial M}$ for any $h\in\Cal G$. Hence these
quantities fit together to a well-defined section
$c(h)\in\Om^{n-1}(\partial M,\Cal T\partial M)$, thus defining a map
$c$ from $\Cal G$ to tractor-valued  differential forms. This is equivariant under
diffeomorphisms preserving $\Cal G$ in the sense that for
$\Phi\in\Diff_{\Cal G}(\barm)$, we obtain
$$
c(\Phi^*h)=(\Phi_\infty)^*c(h).
$$
Here $\Phi_{\infty}:=\Phi|_{\partial M}\in\Conf(\partial M)$ and in the right hand
side we use the standard action of conformal isometries on tractor-valued differential forms.
\end{thm}

\begin{proof}
  Suppose that $g_1,g_2\in\Cal G$ are hyperbolic on $U$. Then we can
  apply Theorem 7.4 of \cite{FeffGr} (see in
    particular the paragraph right after the proof of this theorem in
  \cite{FeffGr}) to their restrictions to $U$. Observe
    also that the additional condition that is needed in Theorem 7.4
    of \cite{FeffGr}, for the case $n=3$, is that the Schouten tensors of
    the two metrics in question agree along the boundary. For the
    rescalings of metrics in $\Cal G$ which extend to the boundary, we
    have verified this in the end of Section \ref{2.5}. This implies
  that there is a neighborhood $V$ of $U\cap\partial M$ in $U$ and a
  diffeomorphism $\Psi:V\to V$ which restricts to the identity on
  $U\cap\partial M$ such that $g_1|_{V}=\Psi^*(g_2|_V)$. Theorem
  \ref{thm3.5} and Corollary \ref{cor3.7} then immediately imply that
  $c(g_1,h)$ and $c(g_2,h)$ coincide on $U\cap\partial M$. It is then
  clear that we obtain the map $c$ as claimed.
 
  The equivariancy of $c$ can be proved locally. So we take $\Phi\in\Diff_{\Cal
    G}(\barm)$ and let $\Phi_\infty$ be its restriction to the boundary. Given
  $x\in\partial M$ we find an open neighborhood $U$ of $x$ in $\barm$ and a metric
  $g\in\Cal G$ such that $g|_U$ is hyperbolic.  Now $\Phi^{-1}(U)$ is an open
  neighborhood of $\Phi^{-1}(x)$ in $\barm$ and $\Phi^*g|_{\Phi^{-1}(U)}$ is
  hyperbolic on $\Phi^{-1}(U)$. Thus we can compute $c(\Phi^*h)$ as
  $c(\Phi^*g,\Phi^*h)$ on $\Phi^{-1}(U)$, and by Proposition \ref{prop3.3} this
  coincides with $(\Phi_\infty)^*(c(g,h)|_U)$. Since $c(g,h)|_U=c(h)|_U$, this implies the
  claim.
\end{proof}

Suppose that $x\in\partial M$ and $U$ is an open neighborhood of $x$
in $\barm$ such that $\Cal G$ contains a metric $g$ which is
hyperbolic on $U$. Then of course the conformal class $[\Cal G_\infty]$ has
to be flat on $U\cap\partial M$. In particular, the assumptions of
Theorem \ref{thm3.8} imply that $(\partial M,[\Cal G_\infty])$ is conformally
flat, which in turn imposes restrictions on $\partial M$. However, if
we are given a manifold $\barm$ with boundary $\partial M$ and a flat
conformal structure on $\partial M$, then there always is a class $\Cal G$
of conformally compact metrics on $M$, for which the assumptions of
Theorem \ref{thm3.8} are satisfied, and hence we obtain an invariant for
single metrics in $\Cal G$.

Indeed, Proposition 7.2 of \cite{FeffGr} (see also the discussion on
p.\ 72 of that reference) shows that there is a hyperbolic metric $g$
on some open neighborhood of $\partial M$ in $\barm$ which induces the
given boundary structure. Then of course $g$ determines an equivalence
class $\Cal G$ of conformally compact ALH-metrics on $M$ for which all
the assumptions of Theorem \ref{thm3.8} are satisfied.

\medskip

We want to point out that it is not clear whether the condition of
conformal flatness in Theorem \ref{thm3.8} is of a fundamental
nature. What one would need in more general situations is a class of
``model metrics'' in $\Cal G$ which can be characterized well enough
to obtain ``uniqueness up to diffeomorphism'' in a form as used in the
proof of Theorem \ref{thm3.8}. An obvious idea is to assume that $\Cal
G$ contains at least one Einstein metric (which is a condition that is
stable under diffeomorphism) and then look at appropriate classes of
Einstein metrics in $\Cal G$. In general, the Einstein condition is
certainly not enough to pin down a metric up to diffeomorphism,
compare with the non-uniqueness issues for the ambient
metric \cite{FeffGr}. However, it is well possible that there are situations in
which additional (geometric) conditions can be imposed to ensure
uniqueness.

\subsection{Recovering mass}\label{3.9}
We now show that in the special case of hyperbolic space that, by an integration
process of our cocycles, we can recover the mass for asymptotically hyperbolic
metrics as introduced by Wang \cite{Wang} and Chru\'{s}ciel-Herzlich
\cite{Chrusciel-Herzlich}. In order to have an appropriate notion of
  integration available, we need the tractor bundle of the boundary to be globally
  trivialized by parallel sections. This essentially means that the boundary has to
  be a sphere. For other topologies, we can distill numerical global invariants out
  of our local invariant, but they cannot be expected to be as strong as a full
  ``integral'' of the local invariant to a tractor, see part (2) of Remark
  \ref{rem3.9} below. So we specialize to the case that $\barm$ is an open
neighborhood of the boundary $S^{n-1}$ in the closed unit ball and that $\Cal G$ is
the equivalence class of (the restriction to $M$ of) the Poincar\'e metric which we
denote by $g$ here. This of course implies that $[\Cal G_\infty]$ is the round
conformal structure on $S^{n-1}$ and that $\Cal G$ satisfies the conditions of
Theorem \ref{thm3.8}. Hence we get a map $c:\Cal G\to \Om^{n-1}(S^{n-1},\Cal
TS^{n-1})$ as described there.

If $n\geq 4$, conformal flatness of the round metric on $S^{n-1}$ implies that the
tractor connection $\nabla^{\Cal T}$ is flat. Moreover, since $(\partial M,[\Cal
  G_\infty])$ is the homogeneous model of conformal structures, the tractor bundle
$\Cal T\partial M$ admits a global trivialization by parallel sections. This extends
to the case $n=3$ with the tractor connection on $S^2$ constructed as discussed in
Section \ref{2.5}. Indeed, since $g$ is conformally flat and Einstein, the ambient
tractor connection is flat and the scale tractor $I^A$ is parallel on all of
$\barm$. In view of Proposition \ref{prop2.6}, $I^A$ thus provides a
  parallel extension of the normal tractor $N^A$ off the boundary, and by \cite{BEG}
  and e.g.\ Lemma 6.2 of \cite{Curry-Gover}, this implies that $\partial M$ is
  totally umbilic in $\barm$. (Recall from Section \ref{2.5} that the trace-free part
  of the second fundamental form is conformally invariant along $\partial M$, so
  being umbilic is a conformally invariant condition.) Hence the second fundamental
form with respect to any metric conformal to $g$ is pure trace. Using a scale with
vanishing mean curvature, as in Section \ref{2.5}, the second fundamental form
actually vanishes. Hence the ambient Levi Civita connection restricts to the Levi
Civita connection on the boundary and by definition we use the restriction of the
ambient Schouten tensor in the construction of the tractor connection on the boundary
in Section \ref{2.5}. Hence formula \eqref{trac-conn} directly implies that the
boundary tractor connection coincides with the restriction of the ambient tractor
connection, so it is flat since the normal tractor is parallel. The trivialization by
parallel sections then works exactly as in higher dimensions. 

This easily implies that, fixing an orientation on $\partial M$, there
is a well-defined integral that associates to each form
$\om\in\Om^{n-1}(\partial M, \Cal T \partial M)$ a parallel section of
$\Cal T\partial M$. Indeed, on $\partial M$ we can take a global frame
$\{s_i\}$ of $\Cal T \partial M$ consisting of parallel sections,
expand $\om$ as $\sum_i\om_is_i$ with $\om_i\in\Om^{n-1}(\partial M)$
and then define $\int_{\partial M}\om:=\sum_i(\int_{\partial M}\om_i)
s_i$. Of course, any other parallel frame consists of linear
combinations of the $s_i$ with constant coefficients, so the result is
independent of the choice of parallel frame. 

Now the boundary tractor metric induces a tensorial map $\Om^{n-1}(\partial
M,\Cal T \partial  M)\x\Ga(\Cal T \partial  M  ) \to\Om^{n-1}(\partial M)$,
which we write as $(\om,s)\mapsto\langle \om,s\rangle$. Observe that
the definition of the integral readily implies that for any parallel
section $s\in\Ga(\Cal T \partial  M)$ on $\partial M$, we obtain $\langle
\int_{\partial M}\om,s\rangle=\int_{\partial
  M}\langle\om,s\rangle$. In particular, the coefficients of
$\int_{\partial M}\om$ with respect to a parallel frame can be
computed as an ordinary integral over an $(n-1)$-form. Given a metric
$h\in\Cal G$, we can in particular apply this to $c(h)$ and we will to
show that, after appropriate normalization, $\int_{\partial M}c(h)$
recovers the mass of $h$.

We will work on $\barm$ with the extension of the conformal class of $g$, which we
again denote by $[g]$. Since $g$ is conformally flat, this leads to a flat tractor
connection and $[g]$ restricts to $[\Cal G_\infty]$ on $\partial M$. Hence any
parallel section $s\in\Ga(\Cal T\partial M)$ can be extended to a parallel section of
$\Cal T\barm$ on a neighborhood of $\partial M$. (In fact, also $\Cal T\barm$ is
globally trivialized by parallel sections, but we don't really need this here.)
Parallel sections of the standard tractor bundle are well understood and
  the approach to the standard tractor bundle in \cite{BEG} is based on their
  relation to the conformal-to-Einstein operator from Remark \ref{rem3.1}. As
  discussed there, the conformal-to-Einstein operator is a conformally invariant
  differential operator which maps $\Ga(\Cal E[1])$ to $\Ga(\Cal E_{(ab)_0}[1])$. It
is well known, see e.g.\ \cite{BEG} that, in terms of the Levi-Civita connection and
the Schouten tensor of any metric in the conformal class, this operator is given by
\begin{equation}\label{conf-Ein}
 \tau\mapsto \nabla_{(a}\nabla_{b)_0}\tau+\Rho_{(ab)_0}\tau
\end{equation}
It is also well known that if $\tau$ lies in the kernel of this operator, then on the
complement of the zero locus of $\tau$, the metric $\tau^{-2}\mathbf{g}_{ab}$
determined by this scale is Einstein. Conversely, for any local Einstein metric in
the conformal class the corresponding (local) scale lies in the kernel. 

We have also discussed in Remark \ref{rem3.1} that the operator \eqref{conf-Ein} can
be realized as a projection of $\nabla_a^{\Cal T}D^A\tau$. So if $D^A\tau$ is a
parallel for $\nabla^{\Cal T}$, the $\tau$ lies in the kernel of
\eqref{conf-Ein}. Now it turns out (compare with \cite{Gover:P-E}) that any parallel
section of the tractor bundle is of the form $D^A\tau$, where $\tau$ is the
projection of the tractor to $\Ga(\Cal E[1])$. Hence this projection and $D^A$
restrict to inverse bijections between parallel standard tractors and sections of
$\Cal E[1]$ that lie in the kernel of \eqref{conf-Ein} (and the zero locus is
automatically nowhere dense and plays no role in this interpretation).

Returning to our setting of the class $[g]$ on $\barm$, we denote by
$\si\in\Ga(\Cal TM)$ the density corresponding to $g$. Since $g$ is
Einstein, $D^A\si$ is a parallel section of $\Cal T\barm|_M$ which
thus extends to all of $\barm$. From Propositions \ref{prop2.5} and
\ref{prop2.6} the orthocomplement of the boundary value can be
naturally identified with $\Cal T\partial M$.  On the other hand, on
$M$, $\si$ is nowhere vanishing. Thus on $M$, we can write any
parallel section of $\Cal T\barm$ as $D^A(V\si)$ for some smooth
function $V:M\to\Bbb R$ such that $V\si$ lies in the kernel of
\eqref{conf-Ein}. But in terms of the Levi-Civita connection of $g$,
for which $\si$ is parallel, and taking into account that the Schouten
tensor is pure trace, this is equivalent to vanishing of the
trace-free part of $\nabla_a\nabla_b V$.  To obtain parallel sections
that lead to boundary tractors along $\partial M$, we can require in
addition that $D^A(V\si)$ is orthogonal to $D^A\si$.  Working in the
scale determined by $g$ and using formula \eqref{trac-met}, one
immediately verifies that this condition is equivalent to requiring
that $V$, in addition, satisfies $\Delta V=-2\Rho V=nV$.  The two
required properties then can be equivalently encoded as a single
equation, the KID (``Killing initial data'') equation (which naturally
decomposes into trace-free part and trace-part)
\begin{equation}\label{KID}
  \nabla_a\nabla_b V-g_{ab}\Delta V+(n-1)g_{ab}V=0. 
\end{equation}
Hence we see that, via $\tfrac1n D^A(V\si)$ on $M$, solutions to this equations
parameterize those parallel tractors on $\barm$ which lie in $\Cal T\partial M$ along
$\partial M$, as recall $\Cal T\partial M$ ~can be identified with
${N_A}^\perp\subset \Cal T \barm |_{\partial M}$. Put another way, solutions to
(\ref{KID}) capture (in a 1-1 way) scales on $M$ whose limit at $\partial M$ is an
Einstein, or almost Einstein (in the sense of \cite{Curry-Gover,Gover:P-E}) metric on
$\partial M$.  But on the other hand, solutions to this equation parameterize the mass
integrals in the classical approach to the AH version of mass, which was our original
motivation for looking for a tractor description.

\begin{thm}\label{thm3.9}
Let $\barm$ be an open neighborhood of $\partial M=S^{n-1}$ in the closed unit ball,
let $\Cal G$ be the equivalence class of the restriction of the hyperbolic
(Poincar\'e) metric $g$ to $M$. For a metric $h\in\Cal G$ consider
$c(h)\in\Om^{n-1}(\partial M,\Cal T\partial M)$ as in Theorem \ref{thm3.8}. For a
solution $V$ of the KID equation \eqref{KID}, let $s_V\in\Ga(\Cal T\partial M)$ be
the parallel section obtained as the boundary value of $\frac{1}{n}
D^A(V\si)\in\Ga(\Cal T\barm)$ (with respect to $[g]$).

Then $-2\langle\int_{\partial M} c(h),s_V\rangle$ coincides with the
mass integral associated to $V$ in \cite{Chrusciel-Herzlich}.
\end{thm}
\begin{proof}
Using the conformal class $[g]$ on $\barm$ we have constructed $c(h)=c(g,h)$ as the
boundary value of $\star_g\al$ for a certain $\Cal T\barm$-valued one-form $\al$ on
$M$. Likewise, for a solution $V$ of \eqref{KID}, the parallel boundary tractor $s_V$
is the restriction to $\partial M$ of a parallel section $\tilde s_V$ of $\Cal
T\barm$. As we have noted already, $\langle\int_{\partial M}
c(h),s_V\rangle=\int_{\partial M}\langle c(h),s_V\rangle$. This integrand is the
boundary value of $\langle\star_g\al,\tilde s_V\rangle$, which equals
$\star_g\langle\al,\tilde s_V\rangle$ by definition. Since this form is smooth up to
the boundary, its integral over $\partial M$ equals the limit as $\ep\to 0$ of the
integrals over the level sets $S_{\ep}=\{x:\rho(x)=\ep\}$. Since the mass integral
associated to $V$ is also expressed via a one-form, it suffices to compare that
one-form to $\langle\al,\tilde s_V\rangle$. In this comparison, we may work up to
terms that vanish along the boundary after application of $\star_g$ and hence up to
$\Cal O(\rho^{n-1})$, cf.\ the proof of Proposition \ref{prop3.1}.
  
We use the description of the mass integral associated to $V$ from
\cite{Michel}, it is shown in that reference that this agrees with the
original mass integral introduces in \cite{Chrusciel-Herzlich}. The
mass integrand associated to a solution $V$ of the KID equation
\eqref{KID} is given by
$$
V(\nabla^i\lambda_{ia}-\nabla_a\tr(\lambda))-g^{ij}\lambda_{ia}\nabla_jV+
\tr(\lambda)\nabla_aV,
$$
where $\la_{ij}=h_{ij}-g_{ij}$, the hyperbolic metric $g$ is used to raise and lower
indices and to form traces, and $\nabla$ is the Levi-Civita connection of $g$.
Decomposing $\lambda_{ij}=\lambda^0_{ij}+\frac{1}{n}\tr(\lambda)g_{ij}$, this becomes
\begin{equation}\label{Michel}
  (V\nabla^i\lambda^0_{ia}-\lambda^0_{ia}\nabla^iV)+
  \tfrac{n-1}n(\tr(\lambda)\nabla_aV-V\nabla_a\tr(\lambda)).  
\end{equation}

Choosing a defining function $\rho$ adapted to $g_{ij}$, we then obtain
$\la^0_{ij}=\rho^{n-2}\mu^0_{ij}$ and $\tr(\la)=\rho^n\mu$ for the quantities
introduced in and below equation \eqref{h-g} (for $N=n$). Note that $\mu^0_{ij}$ and
$\mu$ are smooth up to the boundary. We also know from above that $\si V$ is the
projection of the parallel tractor $\tilde s_V=\frac{1}{n}D(\si V)$, so this is
smooth up to the boundary. Since $\si$ is a defining density for $\partial M$, it
follows that $\rho V$ is smooth up to the boundary.

Now we analyze the two parts of \eqref{Michel} separately, starting with the part
involving $\tr(\la)$. Here we have the advantage that covariant derivatives are only
applied to smooth functions (and not to tensor fields), so the fact that $\nabla$ is
not smooth up to the boundary does not matter. Since $\rho V$ is smooth up to the
boundary,
$\rho\nabla_a\rho V$ is $\Cal
O(\rho)$. Writing $\rho_a$ for $d\rho$ as before, we compute this as $\rho_a\rho
V+\rho^2\nabla_aV$. Thus $\rho^2\nabla_aV=-\rho_a\rho V+\Cal O(\rho)$ and in
particular is smooth up to the boundary. Using this, we obtain
$$
\tr(\lambda)\nabla_aV=\rho^n\mu\nabla_aV=-\rho^{n-2}\rho_a\rho V\mu +\Cal
O(\rho^{n-1}).
$$
Similarly, $V\nabla_a\tr(\la)=V\nabla_a\rho^n\mu=n\rho^{n-2}(\rho
V)\rho_a\mu+O(\rho^{n-1})$. Hence the second part in \eqref{Michel} simply gives
$-\frac{n^2-1}{n}\rho^{n-2}(V\rho)\rho_a\mu+O(\rho^{n-1})$.

Analyzing the second summand in the first part of \eqref{Michel} is similarly
easy. This writes as
$$ -\rho^{-2} g^{ij}\rho^n\mu^0_{ia}\nabla_jV=\rho^{n-2} \rho^{-2}
g^{ij}\rho_j\mu^0_{ia}(\rho V)+\Cal O(\rho^{n-1}).
$$

For the first summand in \eqref{Michel}, the analysis is slightly more
complicated. This can be written
$V\rho\rho^{-2}g^{ij}\rho\nabla_i\rho^{n-2}\mu^0_{ja}$ and hence equals
\begin{equation}\label{tech}
V\rho \rho^{-2}g^{ij}\rho^{n-2}((n-2)\rho_i\mu^0_{ja}+\rho\nabla_i\mu^0_{ja}). 
\end{equation}
Since in the last summand, we apply a covariant derivative to a tensor
field, we have to change to a connection that admits a smooth
extension to the boundary in order to analyze the boundary
behavior. Hence we change from the Levi-Civita connection $\nabla$ of
$g_{ij}$ to the Levi-Civita connection $\bar\nabla$ of $\bar
g_{ij}:=\rho^2g_{ij}$, which has this property. For the usual
conventions, as used in \cite{BEG}, the one form $\Up_a$ associated to
this conformal change is given by $\Up_a=\tfrac{\rho_a}{\rho}$. The
relevant formula for the change of connection is then given by
$$
\nabla_i\mu^0_{ja}=\bar\nabla_i\mu^0_{ja}+2\Up_i\mu^0_{ja}+\Up_j\mu^0_{ia}+
\Up_a\mu^0_{ij}-\Up_k\bar g^{k\ell}\mu^0_{\ell a}\bar g_{ij}-\Up_k\bar
g^{k\ell}\mu^0_{j\ell}\bar g_{ia},  
$$
This immediately shows that $\rho\nabla_i\mu^0_{ja}$ admits a
smooth extension to the boundary and its boundary value can be
obtained by dropping the first summand in the right hand side of this
formula and replacing each occurrence of $\Up$ in the remaining terms
by $d\rho$, so $\Up_i$ becomes $\rho_i$ and so on. Inserting this back
into \eqref{tech}, we get a contraction with $\bar g^{ij}$. This kills
the term involving $\mu^0_{ij}$ by trace-freeness, while all other
terms become multiples of $\bar g^{ij}\rho_i\mu^0_{ja}$. The factors
of the individual terms are $2$, $1$, $-n$, and $-1$, respectively, so
we'll get a total contribution of $(2-n)\bar
g^{ij}\rho_i\mu^0_{ja}$. This actually implies that \eqref{tech} is
$\Cal O(\rho^{n-1})$. Hence we finally conclude that \eqref{Michel}
equals
\begin{equation}\label{Michel-asymp}
  \rho^{n-2}(\rho V)\left(\rho^{-2}g^{ij}\rho_j\mu^0_{ia}-
  \tfrac{n^2-1}n\rho_a\mu\right)+O(\rho^{n-1}).
\end{equation}

Now recall the formula for $\nabla^{\Cal T}_bD^A(\tau-\si)$ from part (1) of Proposition
\ref{prop3.1}, taking into account the definition of $\mathbf{X}^A$. This shows that,
up to $\Cal O(\rho^{n-1})$, $\nabla^{\Cal T}_bD^A(\tau-\si)$ is given by inserting
$\tfrac{n^2-1}2\rho^{n-2}\rho_b\mu\rho\si^{-1}$ into the bottom slot of a
tractor. Pairing this with $\frac{1}{n}D^B(\si V)$ using the tractor metric, we
simply simply obtain the product of $\si V$ with this bottom slot, i.e.\
$$
\tfrac{n^2-1}2\rho^{n-2}\rho_b\mu(\rho V)+\Cal O(\rho^{n-1}).   
$$
Analyzing the formula for $S(\si(h_{ij}-g_{ij})^0)$ from part (1) of Proposition
\ref{prop3.2} we similarly see that the pairing of this with $\frac{1}{n}D^B(\si V)$ 
gives
$$
-\rho^{n-2}\rho^{-2}g^{ij}\rho_j\mu^0_{ia}(\rho V)+\Cal O(\rho^{n-1}). 
$$
But this exactly tells us that, up to $\Cal O(\rho^{n-1})$, \eqref{Michel-asymp}
equals $-2\langle\al,\tilde s_V\rangle$, where $\al$ corresponds to
$c= \frac{1}{n}c_1+\frac{1}{2}c_2$, as in Corollary \ref{cor3.7}.
\end{proof}

  \begin{remark}\label{rem3.9}
(1) The mass integrals in \cite{Chrusciel-Herzlich}, that we compare
    our cocycle to, are known to reproduce the definition of mass by
    Wang in \cite{Wang}, assuming the stronger
    conditions on asymptotics used there. In \cite{Wang} the components of a mass
    vector are obtained from integrals that involve (in our language)
    only the trace of $h$ with respect to $g$. The reason why, in our
    approach, also the trace-free part of the difference of the two
    metrics is needed is explained by Remark \ref{rem3.7}. Indeed the
    basic setup of Wang assumes (by an appropriate choice of
    coordinates) that there is a boundary defining function which is
    adapted both to $h$ and to the background metric $g$.
    
(2) As mentioned at the beginning of Section 3.9, 
    trivialization of the boundary tractor bundle in order to define an integral of
    the invariant $c(h)$ as a parallel section of $\Cal T\partial M$. So this does
    not work for example if the boundary $\partial M$ is a torus, where most of the
    local parallel sections of $\Cal T\partial M$ do not extend to global parallel
    sections. We can form an integral quantity associated to a global parallel
    section of $\Cal T\partial M$ which is the boundary value of the tractor $s_V$
    associated to a solution $V$ of the KID equation as above by forming
    $-2\langle\int_{\partial M} c(h),s\rangle$. These integral quantities are
    sometimes referred to as mass in such situations. However, we want to emphasize
    that they contain only partial information about $c(h)$, since they see only the
    projection of $c(h)$ to the subbundle spanned by parallel tractors. We believe
    that in such situations one should try to work with the full invariant $c(h)$
    rather than just with the integrals against parallel sections. In some cases very
    useful information is likely contained in these ``partial'' mass quantities, but
    understanding this fully we leave as a direction for future research.
\end{remark}

\subsection*{Data availability statement} Data availability is not applicable to this article 
as no new data were created or analyzed in this study.

\end{document}